\documentclass[11pt]{amsart}
\usepackage{bbm}
\usepackage{mathrsfs}
\usepackage{amsfonts}
\usepackage{amssymb}
\usepackage{xypic}
\usepackage{color}

\newtheorem{theorem}{Theorem}[section]
\newtheorem{prop}[theorem]{Proposition}
\newtheorem{cor}[theorem]{Corollary}
\newtheorem{lemma}[theorem]{Lemma}
\newtheorem{definition}[theorem]{Definition}
\newtheorem{remark}[theorem]{Remark}
\newtheorem{example}[theorem]{Example}

\def\Z{{\mathbb{Z}}}
\def\Q{{\mathbb{Q}}}

\def\N{{\mathbb{N}}}
\def\V{{\mathbb{V}}}
\def\A{{\mathbb{A}}}
\def\PP{{\mathbb{P}}}
\def\T{{\mathbb{T}}}
\def\Y{{\mathbb{Y}}}
\def\I{{\mathbb{I}}}

\def\rank{\hbox{\rm{rank}}}
\def\Spec{\hbox{\rm{Spec}}}
\def\Span{\hbox{\rm{Span}}}
\def\Proj{\hbox{\rm{Proj}}}
\def\divsor{\hbox{\rm{div}}}
\def\Div{\hbox{\rm{Div}}}
\def\CDiv{\hbox{\rm{CDiv}}}
\def\Cl{\hbox{\rm{Cl}}}
\def\Pic{\hbox{\rm{Pic}}}
\def\Hom{\hbox{\rm{Hom}}}
\def\image{\hbox{\rm{Im}}}
\def\lc{\hbox{\rm{lc}}}
\def\Syz{\hbox{\rm{Syz}}}
\def\Frac{\hbox{\rm{Frac}}}

\def\p{{P[\sigma]}}
\def\D{{\sigma}}
\def\bu{{\mathbf{u}}}
\def\bw{{\mathbf{w}}}
\def\bv{{\mathbf{v}}}
\def\ba{{\mathbf{a}}}
\def\bb{{\mathbf{b}}}
\def\bc{{\mathbf{c}}}
\def\bd{{\mathbf{d}}}
\def\be{{\mathbf{e}}}
\def\bm{{\mathbf{m}}}

\begin{document}

\title{Toric P-Difference Varieties}
\author{Jie Wang}
\address{KLMM, Academy of Mathematics and Systems Science, Chinese Academy of Sciences, Beijing 100190, China}
\email{wangjie212@mails.ucas.ac.cn}
\date{\today}



\begin{abstract}
In this paper, we introduce the concept of P-difference varieties and study the properties of toric P-difference varieties. Toric P-difference varieties are analogues of toric varieties in difference algebra geometry. The category of affine toric P-difference varieties with toric morphisms is shown to be antiequivalent to the category of affine $P[x]$-semimodules with $P[x]$-semimodule morphisms. Moreover, there is a one-to-one correspondence between the irreducible invariant P-difference subvarieties of an affine toric P-difference variety and the faces of the corresponding affine $P[x]$-semimodule. We also define abstract toric P-difference varieties associated with fans by gluing affine toric P-difference varieties. The irreducible invariant $\p$-subvarieties-faces correspondence is generalized to abstract toric P-difference varieties. By virtue of this correspondence, a divisor theory for abstract toric P-difference varieties is developed.

%\vskip10pt\noindent{\bf Keywords.}
%$\Z[x]$-lattice, $\sigma$-torus, affine $P[x]$-semimodule, $P[\sigma]$-variety, toric $P[\sigma]$-variety, fan, divisor, divisor class module

%\vskip 10pt\noindent{\bf Mathematics Subject Classification [2000]}. {14M25}

\end{abstract}

\maketitle

\section{Introduction}
Toric varieties are very interesting objects of study in algebraic geometry since they have deep connections with the theory of polytopes, symplectic geometry and mirror symmetry, and have applications in many other fields such as physics, coding theory, algebraic statistics and geometric modeling. Toric difference varieties are analogues of toric varieties in difference algebra geometry and are first studied by Gao, Huang, Wang, Yuan in \cite{dd-tdv}. Simply speaking, an affine toric difference variety is an affine difference variety which can be parameterized by difference monomials, or equivalently, is an irreducible affine difference variety containing a difference torus as a Cohn open subset such that the group action of the difference torus on itself extends to a difference algebraic group action on the affine difference variety. In \cite{dd-tdv}, many properties of affine toric difference varieties are characterized by using affine $\N[x]$-semimodules. Actually, the category of affine toric difference varieties with toric morphisms is antiequivalent to the category of affine $\N[x]$-semimodules with $\N[x]$-semimodule morphisms.

In algebraic geometry, the divisor theory is a very useful tool to study the properties of algebraic varieties. However, the divisor theory for toric difference varieties defined in \cite{dd-tdv} does not behave well. In this paper, in order to develop a divisor theory for toric difference varieties, we will introduce another generalization of toric varieties in difference algebra geometry, i.e., toric P-difference varieties.

P-difference varieties are generalizations of ordinary difference varieties by admitting variables of the defining difference polynomials with negative degrees in some sense. More concretely, we define an order on $\mathbb{Z}[x]$ as follows: $f=\sum_{i=0}^pa_ix^i>g=\sum_{i=0}^pb_ix^i$ if and only if there exists an integer $k$ such that $a_i=b_i$ for $k+1\le i\le p$ and $a_k>b_k$. Obviously it is a total order on $\mathbb{Z}[x]$ and $f>0$ if and only if $\lc(f)>0$. Denote $P[x]=\{f\in \mathbb{Z}[x]\mid f\ge0\}$. For $g=\sum_{i=0}^sc_ix^i\in\Z[x]$, denote $a^g=\prod_{i=1}^s(\sigma^i(a))^{c_i}$. Assume $k$ is a difference field with the difference operator $\D$. Let $k\{y_1,\ldots,y_m\}^{P[\sigma]}=k[y_1^{P[x]},\ldots,y_m^{P[x]}]$ which is called the $P[\sigma]$-polynomial ring in the difference variables $y_1,\ldots,y_m$ over $k$, where $y_i^{P[x]}:=\{y_i^g\mid g\in P[x]\}$, $i=1,\ldots,m$. An element in $k\{y_1,\ldots,y_m\}^{P[\sigma]}$ is called a P-difference polynomial. An affine P-difference variety over $k$ is the zero sets defined by some P-difference polynomials. Now we can say an affine toric P-difference variety is an affine P-difference variety which can be parameterized by P-difference monomials. As in the algebraic case, there is a difference algebraic group action on an affine toric P-difference variety. That is to say, an affine toric P-difference variety is an irreducible affine P-difference variety containing a difference torus as an open subset such that the action of the difference torus on itself extends to a difference algebraic group action on the affine P-difference variety.

Every affine toric P-difference variety corresponds to an affine $P[x]$-semimodule, i.e., if $X$ is an affine toric P-difference variety, then there exists an affine $P[x]$-semimodule $S$ such that $X=\Spec^{P[\sigma]}(k[S])$. It turns out that many properties of affine toric P-difference varieties can be described using affine $P[x]$-semimodules. Actually, the category of affine toric P-difference varieties with toric morphisms is antiequivalent to the category of affine $P[x]$-semimodules with $P[x]$-semimodule morphisms. Moreover, there is a one-to-one correspondence between the irreducible invariant P-difference subvarieties of an affine toric P-difference variety and the faces of the corresponding affine $P[x]$-semimodule and a one-to-one correspondence between the $T$-orbits of an affine toric difference variety and the faces of the corresponding affine $P[x]$-semimodule.

A fan is defined to be a finite set of affine $P[x]$-semimodules which satisfies certain compatible conditions. We further define the abstract toric P-difference variety associated with a fan by gluing affine toric P-difference varieties along open subsets. As examples, projective toric P-difference varieties defined by using $\Z[x]$-lattice points are all abstract toric P-difference varieties. The irreducible invariant $\p$-subvarieties-faces correspondence still applies to abstract toric P-difference varieties constructed from fans. By virtue of this correspondence, we can define divisors and divisor class modules for toric P-difference varieties. In particular, the class module and the Picard module of a toric P-difference variety are defined. Moreover, we will establish connections between the properties of toric P-difference varieties and divisor class modules.

The rest of this paper is organized as follows.
In Section 2, we list some preliminaries for difference algebra geometry and preliminaries for $\Z[x]$-lattices which will be used in this paper.
%In section 3, we recall the basic properties of affine toric difference varieties.
In section 3, we introduce the concept of P-difference varieties.
In section 4, affine toric difference varieties are defined and basic properties are proved.
In section 5, projective toric difference varieties are defined and basic properties are proved.
In section 6, we will define abstract toric difference varieties associated with fans and prove their basic properties.
In section 7, we will develop a divisor theory for toric P-difference varieties.
Conclusions are given in Section 8.

\section{Preliminaries}
We list some basic notations and results about difference algebraic geometry and $\Z[x]$-lattices in this section. For more details about difference algebraic geometry, please refer to \cite{Hrushovski1, wibmer}. For more details about $\Z[x]$-lattices, please refer to \cite{dd-tdv}.
\subsection{Preliminaries for Difference Algebraic Geometry}
First we recall some basic notions from difference algebra. For more details, please refer to \cite{levin, wibmer}. All rings in this paper will be assumed to be commutative and unital.

A {\em difference ring} or {\em $\sigma$-ring} for short $(R,\sigma)$, is a ring $R$ together with a ring endomorphism $\sigma\colon R\rightarrow R$. If $R$ is a field, then we call it a {\em difference field}, or a {\em $\sigma$-field} for short. We usually omit $\sigma$ from the notation, simply refer to $R$ as a $\sigma$-ring or a $\sigma$-field. A {\em morphism between $\sigma$-rings} $R$ and $S$ is a ring homomorphism $\psi\colon R\rightarrow S$ which commutes with $\D$. In this paper, all $\sigma$-fields will be assumed to be of characteristic $0$.

Let $k$ be a $\sigma$-field. A $k$-algebra $R$ is called a {\em $k$-$\sigma$-algebra} if the algebra structure map $k\rightarrow R$ is a morphism of $\sigma$-rings. A {\em morphism of $k$-$\sigma$-algebras} is a morphism of $k$-algebras which is also a morphism of $\sigma$-rings. A $k$-subalgebra of a $k$-$\sigma$-algebra is called a {\em $k$-$\sigma$-subalgebra} if it is closed under $\sigma$. If a $k$-$\sigma$-algebra is a $\sigma$-field, then it is called a {\em $\sigma$-field extension} of $k$. Let $R$ and $S$ be two $k$-$\sigma$-algebras. Then $R\otimes_k S$ is naturally a $k$-$\sigma$-algebra by defining $\sigma(r\otimes s)=\sigma(r)\otimes \sigma(s)$ for $r\in R$ and $s\in S$.

Let $k$ be a $\sigma$-field and $R$ a $k$-$\sigma$-algebra. For a subset $A$ of $R$, the smallest $k$-$\sigma$-subalgebra of $R$ containing $A$ is denoted by $k\{A\}$. If there exists a finite subset $A$ of $R$ such that $R=k\{A\}$, we say that $R$ is {\em finitely $\sigma$-generated} over $k$. If additionally $R$ is a $\sigma$-field, the smallest $k$-$\sigma$-subfield of $R$ containing $A$ is denoted by $k\langle A\rangle$.

Now we introduce the following useful notation. Let $x$ be an algebraic indeterminate and $p=\sum_{i=0}^s c_i x^i\in\Z[x]$. For $a$ in a $\sigma$-field $K$, denote $a^p = \prod_{i=0}^s (\sigma^i(a))^{c_i}$. It is easy to check that for $p, q\in\Z[x], a^{p+q}=a^{p} a^{q}, a^{pq}= (a^{p})^{q}$.

Let $k$ be a $\sigma$-field. Suppose $y=\{y_1,\ldots,y_m\}$ is a set of $\sigma$-indeterminates over $k$. Then the {\em $\sigma$-polynomial ring} over $k$ in $y$ is the polynomial ring in the variables $y,\sigma(y),\sigma^2(y),\ldots$. It is denoted by $k\{y_1,\ldots,y_m\}$ and has a natural $k$-$\sigma$-algebra structure. An element in $k\{y_1,\ldots,y_m\}$ is called a {\em $\sigma$-polynomial} over $k$. A {\em $\sigma$-polynomial ideal}, or simply a {\em $\sigma$-ideal}, $I$ in $k\{y_1,\ldots,y_m\}$ is an algebraic ideal which is closed under $\sigma$, i.e.\ $\sigma(I)\subset I$. If $I$ also has the property that $\sigma(a)\in I$ implies $a\in I$, it is called a {\em reflexive $\sigma$-ideal}. A {\em $\sigma$-prime ideal} is a reflexive $\sigma$-ideal which is prime as an algebraic ideal. A $\sigma$-ideal $I$ is called {\em perfect} if for any $g\in\N[x]\setminus\{0\}$ and $a\in k\{y_1,\ldots,y_m\}$, $a^g\in I$ implies $a\in I$. It is easy to prove every $\sigma$-prime ideal is perfect. If $S$ is a finite set of $\sigma$-polynomials in $k\{y_1,\ldots,y_m\}$, we use $(S)$, $[S]$, and $\{S\}$ to denote the algebraic ideal, the $\sigma$-ideal, and the perfect $\sigma$-ideal generated by $S$ respectively.

Let $k$ be a $\sigma$-field. We denote the category of $\sigma$-field extensions of $k$ by $\mathscr{E}_k$ and the category of $K^n$ by $\mathscr{E}_k^n$ where $K \in\mathscr{E}_k$. Let $F\subseteq k\{y_1,\ldots,y_m\}$ be a set of $\sigma$-polynomials. For any $K \in\mathscr{E}_k$, define the solutions of $F$ in $K$ to be
$$\mathbb{V}_K(F):=\{a\in K^n\mid f(a)=0 \textrm{ for all } f\in F\}.$$
Note that $K\rightsquigarrow \V_K(F)$ is naturally a functor from the category of $\sigma$-field extensions of $k$ to the category of sets. Denote this functor by $\V(F)$.
\begin{definition}
Let $k$ be a $\sigma$-field. An {\em (affine) difference variety} or {\em $\sigma$-variety} over $k$ is a functor $X$ from the category of $\sigma$-field extensions of $k$ to the category of sets which is of the form $\mathbb{V}(F)$ for some subset $F$ of $k\{y_1,\ldots,y_m\}$. In this situation, we say that $X$ is the (affine) $\sigma$-variety defined by $F$.
\end{definition}

If no confusion is caused, we will omit the word ``affine'' for simplicity.

The functor $\mathbb{A}_k^m$ given by $\mathbb{A}_k^m(K)=K^m$ for $K\in \mathscr{E}_k$ is called the {\em $\sigma$-affine ($m$-)space} over $k$. Obviously, $\mathbb{A}_k^m=\mathbb{V}(0)$ is an affine $\sigma$-variety over $k$.

By definition, a {\em morphism} $\phi\colon X\rightarrow Y$ of $\sigma$-varieties consists of maps $\phi_K\colon X(K)\rightarrow Y(K)$ for any $K \in\mathscr{E}_k$.
If $X$ and $Y$ are two $\sigma$-varieties over $k$, then we write $X\subseteq Y$ to indicate that $X$ is a subfunctor of $Y$. This simply means that $X(K)\subseteq Y(K)$ for every $K\in \mathscr{E}_k$. In this situation, we also say that $X$ is a {\em $\sigma$-subvariety} of $Y$.

Let $X$ be a $\sigma$-subvariety of $\mathbb{A}_k^m$. Then
\[\mathbb{I}(X):=\{f\in k\{y_1,\ldots,y_m\}\mid f(a)=0\textrm{ for all }a\in X(K) \textrm{ and all }K\in \mathscr{E}_k\}\]
is called the {\em vanishing ideal} of $X$.
It is well known that $\sigma$-subvarieties of $\mathbb{A}_k^m$ are in a one-to-one correspondence with perfect $\sigma$-ideals of $k\{y_1,\ldots,y_m\}$ and we have $\I(\V(F))=\{F\}$ for $F\subseteq k\{y_1,\ldots,y_m\}$.

\begin{definition}
Let $X$ be a $\sigma$-subvariety of $\mathbb{A}_k^m$. Then the $k$-$\sigma$-algebra
\[k\{X\}:=k\{y_1,\ldots,y_m\}/\mathbb{I}(X)\]
is called the {\em $\sigma$-coordinate ring} of $X$.
\end{definition}

A $k$-$\sigma$-algebra is called an {\em affine $k$-$\sigma$-algebra} if it is isomorphic to $k\{y_1,\ldots,y_m\}/\mathbb{I}(X)$ for some affine $\D$-variety $X$. Then by definition, $k\{X\}$ is an affine $k$-$\sigma$-algebra.

The following lemma is taken from \cite[Remark 2.1.10]{wibmer}.
\begin{lemma}\label{pdag-lemma1}
Let $X$ be a $k$-$\sigma$-variety. Then for any $K\in \mathscr{E}_k$, there is a natural bijection between $X(K)$ and the set of $k$-$\sigma$-algebra homomorphisms from $k\{X\}$ to $K$. Indeed,
$$X\simeq \Hom(k\{X\},\underline{\mspace{18mu}})$$
as functors.
\end{lemma}
\begin{definition}
Let $X\subseteq \mathbb{A}_k^m$ and $Y\subseteq \mathbb{A}_k^n$ be $k$-$\sigma$-varieties. A morphism of functors $f\colon X\rightarrow Y$ is called a {\em morphism of $k$-$\sigma$-varieties} if there exist $\sigma$-polynomials $f_1,\ldots,f_n\in k\{y_1,\ldots,y_m\}$ such that $f(a)=(f_1(a),\ldots,f_n(a))$ for every $a\in X(K)$ and all $K\in \mathscr{E}_k$.
\end{definition}

Similarly to affine algebraic varieties, we have
\begin{theorem}
Let $k$ be a $\sigma$-field. The category of affine $k$-$\sigma$-varieties is antiequivalent to the category of affine $k$-$\sigma$-algebras.
\end{theorem}
\begin{proof}
Please refer to \cite[2.1]{wibmer}.
\end{proof}

Suppose $X$ is an affine $k$-$\sigma$-variety. Let $\Spec^{\sigma}(k\{X\})$ be the set of all $\sigma$-prime ideals of $k\{X\}$. Let $F\subseteq k\{X\}$. Set
$$\mathcal{V}(F):=\{\mathfrak{p}\in \Spec^{\sigma}(k\{X\})\mid F\subseteq \mathfrak{p}\}\subseteq \Spec^{\sigma}(k\{X\}).$$
Obviously, $\mathcal{V}(F)=\mathcal{V}(\{F\})$. It can be checked that $\Spec^{\sigma}(k\{X\})$ is a topological space with closed sets of forms $\mathcal{V}(F)$. Then the {\em topological space} of $X$ is $\Spec^{\sigma}(k\{X\})$ equipped with the above Cohn topology.

Let $k$ be a $\sigma$-field and $F\subseteq k\{y_1,\ldots,y_m\}$. Let $K,L\in\mathscr{E}_k$. Two solutions $a\in\V_K(F)$ and $b\in\V_L(F)$ are called {\em equivalent} if there exists a $k$-$\sigma$-isomorphism between $k\langle a\rangle$ and $k\langle b\rangle$ which maps $a$ to $b$. Obviously this defines an equivalence relation.
The following theorem gives a relationship between equivalence classes of solutions of $I$ and $\sigma$-prime ideals containing $I$. For the proof, please refer to \cite[p.31]{wibmer}.
\begin{theorem}\label{pdag-thm}
Let $X$ be a $k$-$\sigma$-variety. There is a natural bijection between the set of equivalence classes of solutions of $\mathbb{I}(X)$ and $\Spec^{\sigma}(k\{X\})$.
\end{theorem}

Because of Theorem \ref{pdag-thm}, we shall not strictly distinguish between a $\sigma$-variety and its topological space. In other words, we use $X$ to mean the $\sigma$-variety or its topological space.

%The following two propositions relating morphisms between $\sigma$-varieties and morphisms between their $\sigma$-coordinate rings will be useful. See \cite{wibmer} p.32 and p.33.
%\begin{prop}
%Let $f\colon X\rightarrow Y$ be a morphism of $k$-$\sigma$-varieties. Then $f$ is dominant if and only if the dual map $f^*\colon k\{y_1,\ldots,y_m\}\rightarrow k\{X\}$ is injective.
%\end{prop}
%\begin{prop}\label{padg-prop}
%Let $f\colon X\rightarrow Y$ be a morphism of $k$-$\sigma$-varieties. Then $f$ is a closed embedding if and only if the dual map $f^*\colon k\{y_1,\ldots,y_m\}\rightarrow k\{X\}$ is surjective.
%\end{prop}

\subsection{Preliminaries for $\Z[x]$-lattices}
A $\Z[x]$-module which can be embedded into $\Z[x]^n$ for some $n$ is called a {\em $\Z[x]$-lattice}. Since $\Z[x]^n$ is Noetherian as a $\Z[x]$-module, we see that any $\Z[x]$-lattice is finitely generated. Let $L$ be a $\Z[x]$-lattice. We always identity it with a $\Z[x]$-submodule of $\Z[x]^n$ for some $n$. Define the {\em rank} of $L$ to be
\[\rank(L):=\dim_{\Q(x)}\Span_{\Q(x)}(L).\]
Note that $L$ may not be a free $\Z[x]$-module, thus the number of minimal generators of $L$ can be larger than its rank.

Sometimes we want to know whether a $\Z[x]$-module is a $\Z[x]$-lattice, i.e.\ whether it can be embedded into $\Z[x]^n$ for some $n$. The following lemma is taken from \cite[p.172]{rotman}:
\begin{lemma}\label{pzl-lemma}
Let $R$ be a domain and $A$ an $R$-module. If $A$ is finitely generated and torsion-free, then $A$ can be imbedded into a finitely generated free $R$-module.
\end{lemma}

Therefore, the condition for a finitely generated $\Z[x]$-module to be a $\Z[x]$-lattice is that it has no torsion.

Suppose $U=\{\bu_1,\ldots,\bu_m\}\subset\Z[x]^n$. The {\em syzygy module} of $U$, which is denoted by $\Syz(U)$, is $$\Syz(U):=\{\bv\in\Z[x]^m\mid U\bv=\mathbf{0}\},$$
where we regard $U$ as a matrix with columns $\bu_i$.

It is clear that $\Syz(U)$ is a $\Z[x]$-lattice in $\Z[x]^m$. Moreover,
\begin{lemma}\label{pzl-lemma2}
$\Syz(U)$ is a free $\Z[x]$-module of rank $m-\rank(U)$.
\end{lemma}

A $\Z[x]$-lattice $L\subseteq \Z[x]^{m}$ is said to be {\em toric} if it is $\Z[x]$-saturated, that is for any nonzero $g\in\Z[x]$ and $\bu \in\Z[x]^{m}$, $g\bu \in L$ implies $\bu \in L$.
\begin{remark}\label{pi-re1}
If $U=\{\bu_1,\ldots,\bu_m\}\subset \Z[x]^n$, then the syzygy module $L$ of $U$ is obviously $\Z[x]$-saturated and hence toric.
\end{remark}

For a $\Z[x]$-lattice $L\subseteq \Z[x]^m$, let $$L^C:=\{\bu\in \Z[x]^m\mid\langle \bu,\bv\rangle=0, \forall\bv\in L\},$$ where $\langle \bu,\bv\rangle=\bu^{\tau}\bv$ is the dot product of $\bu$ and $\bv$. By Lemma \ref{pzl-lemma2}, $L^C$ is a free $\Z[x]$-module and of rank $m-\rank(L)$.
\begin{remark}\label{pi-re2}
For a toric $\Z[x]$-lattice, one can check that $(L^C)^C=L$.
\end{remark}

For $\bu=(u_1,\ldots,u_m)\in \Z[x]^m$, we denote $\Y^{\bu}=\prod_{i=1}^my_i^{u_i}$. $\Y^{\bu}$ is called a {\em $\sigma$-monomial} in $\Y$ and ${\bu}$ is called its {\em support}.
\begin{definition}
Given a $\Z[x]$-lattice $L\subseteq \Z[x]^m$, we define a binomial $\sigma$-ideal $I_L\subseteq k\{y_1,\ldots,y_m\}$ associated with $L$:
\[I_L:=[\Y^{\bu^+}-\Y^{\bu^-}\mid\bu\in L]=[\Y^{\bu}-\Y^{\bv}\mid \bu,\bv\in \N[x]^m \textrm{ with } \bu-\bv\in L],\]
where $\bu^{+}, \bu^{-}\in \N[x]^m$ are the
positive part and the negative part of $\bu=\bu^{+}-\bu^{-}$, respectively. $L$ is called the {\em support lattice} of $I_L$.
If $L$ is toric, then the corresponding $\Z[x]$-lattice ideal $I_L$ is called a {\em toric $\sigma$-ideal}.
\end{definition}

The following two lemmas will be used later.
\begin{lemma}\label{pi-le}
Let $M$ be a $\Z[x]$-lattice. Then $N=\Hom_{\Z[x]}(M,\Z[x])$ is a free $\Z[x]$-module and has the same rank with $M$.
\end{lemma}
\begin{proof}
Suppose $M=\Z[x](\{\bu_1,\ldots,\bu_m\})\in\Z[x]^n$. Define a map
$$\theta\colon\Z[x]^m\longrightarrow M, \be_i\longmapsto \bu_i,$$
where $\{\be_i\}_{i=1}^m$ is the standard basis of $\Z[x]^m$. Let $L=\ker(\theta)$. By Lemma \ref{pzl-lemma2}, $\rank(L)=m-\rank(M)$. Define a map
$$\alpha\colon N\rightarrow \Z[x]^m, \alpha(\varphi)=(\varphi(\bu_1),\ldots,\varphi(\bu_m)), \forall\varphi\in N.$$
It is easy to see that $\alpha$ is an embedding and the image of $\alpha$ is $L^C$ which implies $N\simeq L^C$. Hence $N$ is free and $\rank(N)=m-\rank(L)=m-(m-\rank(M))=\rank(M)$ by Lemma \ref{pzl-lemma2}.
\end{proof}

For $\bv\in \Z[x]^n$, we define a $\Z[x]$-module homomorphism $\varphi_{\bv}\colon M\rightarrow \Z[x]$ by $\varphi_{\bv}(\bu)=\langle \bu,\bv\rangle$, for all $\bu\in M$. So $\varphi_{\bv}\in N$ and we get a map $\theta\colon \Z[x]^n\rightarrow N, \bv\mapsto \varphi_{\bv}$.
\begin{lemma}\label{atv-lemma}
For any $\varphi\in N$, there exists $g\in \Z[x]$ such that $g\varphi=\varphi_{\bv}$ for some $\bv\in \Z[x]^n$.
\end{lemma}
\begin{proof}
The map $\theta$ gives an exact sequence:
\[0\longrightarrow M^C\longrightarrow\Z[x]^n\stackrel{\theta}{\longrightarrow} N.\]
Tensor it with $\Q(x)$ to obtain:
\[0\longrightarrow M^C_{\Q(x)}\longrightarrow\Q(x)^n\stackrel{\theta_{\Q(x)}}{\longrightarrow} N_{\Q(x)}.\]
Therefore $\rank(\image(\theta_{\Q(x)}))=n-\rank(M^C_{\Q(x)})=n-(n-\rank(M))=\rank(M)$. It follows $\rank(\image(\theta))=\rank(M)$. Since $N$ is a free $\Z[x]$-module by Lemma \ref{pi-le}, for any $\varphi\in N$, there exists $g\in \Z[x]$ such that $g\varphi\in \image(\theta)$, which implies that there exists $\bv\in\Z[x]^n$ such that $g\varphi=\varphi_{\bv}$.
\end{proof}

\section{Affine $P[\sigma]$-Varieties}
In this section, we will introduce the concept of affine $P[\sigma]$-varieties which is a generalization of the usual $\sigma$-variety.

\subsection{Perfect $P[\sigma]$-Ideals}
Let us define an order on $\mathbb{Z}[x]$ as follows: $f=\sum_{i=0}^pa_ix^i>g=\sum_{i=0}^pb_ix^i$ if and only if there exists an integer $k$ such that $a_i=b_i$ for $k+1\le i\le p$ and $a_k>b_k$. Obviously it is a total order on $\mathbb{Z}[x]$ and $f>0$ if and only if $\lc(f)>0$. Denote $P[x]=\{f\in \mathbb{Z}[x]\mid f\ge0\}$ and $P[x]^*=P[x]\backslash{\{0\}}$. In a $\sigma$-field $K$, we set that for $a\in K, g\in P[x]^*, a^g=0$ if and only if $a=0$.

Let $k$ be a $\sigma$-field. Suppose $y=\{y_1,\ldots,y_m\}$ is a set of $\sigma$-indeterminates over $k$. Then the $P[\sigma]$-polynomial ring over $k$ in $y$ is the polynomial ring in the variables $y_1^{P[x]},\ldots,y_m^{P[x]}$, where $y_i^{P[x]}$ means $\{y_i^g\mid g\in P[x]^*\}, i=1,\ldots,m$. It is denoted by
\[k\{y_1,\ldots,y_m\}^{\p}:=k[y_1^{P[x]},\ldots,y_m^{P[x]}]\]
and has a natural $k$-$\sigma$-algebra structure. An element in $k\{y_1,\ldots,y_m\}^{\p}$ is called a {\em $P[\sigma]$-polynomial} over $k$. For $\bu=(u_1,\ldots,u_m)\in P[x]^m$, $\Y^{\bu}=\prod_{i=1}^my_i^{u_i}$ is called a {\em $P[\sigma]$-monomial}.
A {\em $\p$-term} is the product of a constant in $k$ and a $\p$-monomial.

\begin{definition}
A $\sigma$-ideal $I\subseteq k\{y_1,\ldots,y_m\}^{P[\sigma]}$ is called a {\em $P[\sigma]$-ideal} if $\Y^{\bu}f\in I$ implies $\Y^{g\bu}f\in I$ for any $g\in P[x]^*$, where $\Y^{\bu}$ is a $P[\sigma]$-monomial.\\
A $\sigma$-ideal $I\subseteq k\{y_1,\ldots,y_m\}^{P[\sigma]}$ is called a {\em $P[\sigma]$-perfect ideal} if it is a perfect $\sigma$-ideal and a $P[\sigma]$-ideal.\\
A $\sigma$-ideal $I\subseteq k\{y_1,\ldots,y_m\}^{P[\sigma]}$ is called a {\em $P[\sigma]$-prime ideal} if it is a $P[\sigma]$-perfect ideal and a prime ideal.
\end{definition}
\begin{remark}\label{pi-remark}
For $\bu\in P[x]^m$, we denote $b_{\bu}=(b_1,\ldots,b_m)\in\{0,1\}^m$ such that $b_i=1$ if $u_i\ne0$ and $b_i=0$ if $u_i=0$ for $i=1,\ldots,m$. If $I\subseteq k\{y_1,\ldots,y_m\}^{P[\sigma]}$ is a perfect $\sigma$-ideal and $\Y^{\bu}f\in I$ with $\bu=\bu^+-\bu^-$, where $\bu^+,\bu^-\in \N[x]^s$, then $\Y^{\bu^+}f\in I$, and therefore by the property of perfect $\sigma$-ideals, $\Y^{b_{\bu}}f\in I$. Furthermore, if $I$ is a $\p$-perfect ideal, then $\Y^{\bu}f\in I$ implies that $\Y^{\bv}f\in I$, for any $\bv\in P[x]^m$ satisfying $v_i\ne0$ if $u_i\ne0, 1\le i\le m$.
\end{remark}

It is easy to check that the intersection of $\p$-perfect ideals is again a $\p$-perfect ideal. Therefore, each subset $F$ of $k\{y_1,\ldots,y_m\}^{P[\sigma]}$ is contained in a smallest $\p$-perfect ideal, which is called the {\em $\p$-perfect closure} of $F$ or the $\p$-perfect ideal generated by $F$. It is denoted by $\{F\}^{\p}$.

For a perfect $\sigma$-ideal $\mathfrak{a}$ of $k\{y_1,\ldots,y_m\}^{P[\sigma]}$, set
\begin{align}\label{pi-equ1}
\mathfrak{a}':&=\{\Y^{\bu}f\mid\Y^{b_{\bu}}f\in\mathfrak{a}, \bu\in P[x]^m\}\\
&=\{y_{r_1}^{u_1}\cdots y_{r_s}^{u_s}f\mid y_{r_1}\cdots y_{r_s}f\in\mathfrak{a}, u_i\in P[x]^*, 1\le i\le s\}.
\end{align}
Let $F$ be a subset of $k\{y_1,\ldots,y_m\}^{P[\sigma]}$. Define $F^{[1]}=\{F\}'$ and recursively define $F^{[i]}=\{F^{[i-1]}\}'$ for $i\geqslant 2$. One can check that $\{F\}^{\p}=\cup_{i\geqslant 1}F^{[i]}$. Moreover, we have
\begin{lemma}
Let $F$ and $G$ be two subsets of $k\{y_1,\ldots,y_m\}^{P[\sigma]}$. Then
\begin{enumerate}
\item[(a)] $F^{[1]}G^{[1]}\subseteq (FG)^{[1]}$;
\item[(b)] $F^{[i]}G^{[i]}\subseteq (FG)^{[i]}$ for $i\geqslant 1$;
\item[(c)] $F^{[i]}\cap G^{[i]}=(FG)^{[i]}$ for $i\geqslant 1$.
\end{enumerate}
\end{lemma}
\begin{proof}
(a): Let $\Y^{\bu}f\in F^{[1]}$ and $\Y^{\bv}g\in G^{[1]}$. Then by (\ref{pi-equ1}), $\Y^{b_{\bu}}f\in \{F\}$ and $\Y^{b_{\bv}}g\in \{G\}$. So $\Y^{b_{\bu}}f\cdot\Y^{b_{\bv}}g\in \{F\}\cap\{G\}=\{FG\}$, and hence $\Y^{\bu}f\cdot\Y^{\bv}g\in (FG)^{[1]}$.\\
(b): We prove (b) by induction on $i$. The case $i=1$ is proved by (a). For the inductive step, assume now $i\ge2$. Then by (a) and the induction hypothesis,
\begin{align*}
  F^{[i]}G^{[i]}&=(F^{[i-1]})^{[1]}(G^{[i-1]})^{[1]}\subseteq (F^{[i-1]}G^{[i-1]})^{[1]}\\
  &\subseteq ((FG)^{[i-1]})^{[1]}=(FG)^{[i]}.
\end{align*}
(c): It is obvious that $(FG)^{[i]}\subseteq F^{[i]}\cap G^{[i]}$. For the converse, let $\Y^{\bu}f\in F^{[i]}\cap G^{[i]}$, then $\Y^{2\bu}f^2\in F^{[i]}G^{[i]}\subseteq (FG)^{[i]}=\{(FG)^{[i-1]}\}'$, so by (\ref{pi-equ1}), $\Y^{b_{\bu}}f^2\in \{(FG)^{[i-1]}\}$ and hence $\Y^{b_{\bu}}f\in \{(FG)^{[i-1]}\}$. It follows $\Y^{\bu}f\in (FG)^{[i]}$ which proves $F^{[i]}\cap G^{[i]}\subseteq(FG)^{[i]}$.
\end{proof}
\begin{prop}\label{pi-prop}
Let $F$ and $G$ be two subsets of $k\{y_1,\ldots,y_m\}^{P[\sigma]}$. Then
\[\{F\}^{\p}\cap\{G\}^{\p}=\{FG\}^{\p}.\]
\end{prop}
\begin{proof}
Obviously, $\{F\}^{\p}\cap\{G\}^{\p}\supseteq\{FG\}^{\p}$. For the converse, let $f\in \{F\}^{\p}\cap\{G\}^{\p}$, then there exist $i$ and $j$ such that $f\in \{F\}^{[i]}$ and $f\in \{G\}^{[j]}$. Without loss of generality, we can assume $i\le j$, then $f\in \{F\}^{[j]}$. Therefore, $f\in F^{[j]}\cap G^{[j]}=(FG)^{[j]}\subseteq\{FG\}^{\p}$.
\end{proof}
\begin{theorem}\label{pi-thm1}
Let $F$ be a subset of $k\{y_1,\ldots,y_m\}^{P[\sigma]}$. Then $\{F\}^{\p}$ is the intersection of all $\p$-prime ideals containing $F$. In particular, every $\p$-perfect ideal is the intersection of $\p$-prime ideals.
\end{theorem}
\begin{proof}
Because $\p$-prime ideals are $\p$-perfect, it is clear that $\{F\}^{\p}$ is contained in every $\p$-prime ideal containing $F$. It suffices to show that every $\p$-perfect ideal is the intersection of $\p$-prime ideals.

Let $I$ be any $\p$-perfect ideal of $k\{y_1,\ldots,y_m\}^{P[\sigma]}$. If $f$ is contained in every $\p$-prime ideal containing $I$, we have to show that $f\in I$. Suppose the contrary, then we just need to find a $\p$-prime ideal containing $I$ which doesn't contain $f$. Let $\Sigma=\{\mathfrak{p}\mid\textrm{ $\mathfrak{p}$ is a $\p$-perfect ideal}, \mathfrak{p}\supseteq I, f\notin \mathfrak{p}\}$. Clearly the union of an ascending chain of $\p$-perfect ideals containing $I$ not containing $f$ is again a $\p$-perfect ideal containing $I$ not containing $f$. Since $I\in\Sigma$, $\Sigma$ is a nonempty set. So by Zorn's Lemma, there exists a maximal element in $\Sigma$, denoted it by $\mathfrak{q}$. We claim that $\mathfrak{q}$ is a $\p$-prime ideal.

We only need to show that $\mathfrak{q}$ is prime. Suppose $gh\in \mathfrak{q}$, if both $g$ and $h$ are not in $\mathfrak{q}$, then by the maximality of $\mathfrak{q}$, $f\in \{\mathfrak{q},g\}^{\p}$ and $f\in \{\mathfrak{q},h\}^{\p}$. So $f\in \{\mathfrak{q},g\}^{\p}\cap\{\mathfrak{q},h\}^{\p}$. By Proposition \ref{pi-prop}, $\{\mathfrak{q},g\}^{\p}\cap\{\mathfrak{q},h\}^{\p}=\{\mathfrak{q}^2,\mathfrak{q}g,\mathfrak{q}h,gh\}^{\p}\subseteq \mathfrak{q}$, so $f\in \mathfrak{q}$ which is contradictory to the choice of $\mathfrak{q}$. Thus $\mathfrak{q}$ is a $\p$-prime ideal and it contains $I$ but doesn't contain $f$ as desired.
\end{proof}
\begin{remark}
It is well known that every perfect $\sigma$-ideal is a finite intersection of $\sigma$-prime ideals, which is equivalent to the finite generated property of perfect $\sigma$-ideals. However, in the $\p$-case, it is still unknown that if a $\p$-perfect ideal could be written as an intersection of finitely many $\p$-prime ideals.
\end{remark}

\subsection{Affine $P[\sigma]$-Varieties}
Let $F$ be any subset of $k\{y_1,\ldots,y_m\}^{P[\sigma]}$ and $K$ any $\sigma$-field extension of $k$. We define the solutions of $F$ in $K$ to be
$$\mathbb{V}_K(F)=\{a\in K^n\mid f(a)=0 \textrm{ for all } f\in F\}.$$
\begin{definition}
Let $k$ be a $\sigma$-field. An {\em affine $P[\sigma]$-variety} over $k$ is a functor $X$ from the category of $\sigma$-field extensions of $k$ to the category of sets which is of the form $X=\mathbb{V}(F)$ for some subset $F$ of $k\{y_1,\ldots,y_m\}^{P[\sigma]}$. In this situation, we say that $X$ is the (affine) $P[\sigma]$-variety defined by $F$.
\end{definition}

Obviously, the $\sigma$-affine space $\mathbb{A}_k^m=\mathbb{V}(0)$ is an affine $P[\sigma]$-variety over $k$. Moreover, every affine $\sigma$-variety can be naturally viewed as an affine $\p$-variety.

If $X$ and $Y$ are two $P[\sigma]$-varieties over $k$, then we write $X\subseteq Y$ to indicate that $X$ is a subfunctor of $Y$. This simply means that $X(K)\subseteq Y(K)$ for every $K\in \mathscr{E}_k$. In this situation, we also say that $X$ is a {\em $P[\sigma]$-subvariety} of $Y$.

Let $X$ be a $P[\sigma]$-subvariety of $\mathbb{A}_k^m$. Then
\[\mathbb{I}(X):=\{f\in k\{y_1,\ldots,y_m\}^{P[\sigma]}\mid f(a)=0\textrm{ for all }K\in \mathscr{E}_k\textrm{ and all }a\in X(K)\}\]
is called the {\em vanishing ideal} of $X$.
\begin{lemma}\label{apv-lemma}
Let $X$ be a $P[\sigma]$-subvariety of $\mathbb{A}_k^m$. Then $\mathbb{I}(X)$ is a $P[\sigma]$-perfect ideal.
\end{lemma}
\begin{proof}
Clearly, $\mathbb{I}(X)$ is a perfect $\sigma$-ideal. Suppose that $\Y^{\bu}f \in I$ and $g\in P[x]^*$. Then for every $K\in \mathscr{E}_k$ and for every $a\in X(K)$, $a^{\bu}f(a)=0$, which implies $a^{g\bu}f(a)=0$. It follows $\Y^{g\bu}f\in I$. Thus $\mathbb{I}(X)$ is a $P[\sigma]$-perfect ideal.
\end{proof}
\begin{definition}
Let $X$ be a $P[\sigma]$-subvariety of $\mathbb{A}_k^m$. Then the $k$-$\sigma$-algebra
\[k\{X\}:=k\{y_1,\ldots,y_m\}^{P[\sigma]}/\mathbb{I}(X)\]
is called the $P[\sigma]$-coordinate ring of $X$.
\end{definition}

A $k$-$\sigma$-algebra isomorphic to some $k\{y_1,\ldots,y_m\}^{P[\sigma]}/\mathbb{I}(X)$ is called an {\em affine $k$-$P[\sigma]$-algebra}.
\begin{lemma}\label{apv-lemma1}
Let $X$ be a $k$-$P[\sigma]$-variety. Then for any $K\in \mathscr{E}_k$, there is a natural bijection between $X(K)$ and the set of $k$-$\sigma$-algebra homomorphisms from $k\{X\}$ to $K$. Indeed, $X\simeq \Hom(k\{X\},\underline{\mspace{18mu}})$ as functors.
\end{lemma}
\begin{proof}
Let $\bar{y}$ to denote the coordinate functions on $X$. If $a\in X(K)$, define a $k$-$\sigma$-morphism from $k\{X\}$ to $K$ by sending $\bar{y}$ to $a$. Conversely, given a $k$-$\sigma$-morphism $\phi\colon k\{X\}\rightarrow K$, then $a:=\phi(\bar{y})$ belongs to $X(K)$.
\end{proof}
\begin{prop}\label{prop}
Let $F$ be a subset of $k\{y_1,\ldots,y_m\}^{\p}$, then
\[\I(\V(F))=\{F\}^{\p}.\]
\end{prop}
\begin{proof}
Clearly, $F\subset \I(\V(F))$. Since by Lemma \ref{apv-lemma}, $\I(\V(F))$ is a $\p$-perfect ideal, we have $\{F\}^{\p}\subseteq \I(\V(F))$. Now suppose $f\in \I(\V(F))$, we need to show $f\in \{F\}^{\p}$. Since $\{F\}^{\p}$ is the intersection of all $\p$-prime ideals containing $F$, it suffices to show that $f$ lies in every $\p$-prime ideal $\mathfrak{p}$ with $F\subset \mathfrak{p}$. Let $K$ be the residue class field of $\mathfrak{p}$ and let $a$ be the image of $y=\{y_1,\ldots,y_m\}$ in $K$. Since $F\subset \mathfrak{p}$, $a\in \V_K(F)$. And $f\in \I(\V(F))$ implies $f(a)=0$. Therefore $f\in \mathfrak{p}$.
\end{proof}
\begin{theorem}\label{apv-thm1}
The maps $X\mapsto \I(X)$ and $I\mapsto \V(I)$ define inclusion reversing bijections between the set of all $\p$-subvarieties of $\A_k^m$ and the set of all $\p$-perfect ideals of $k\{y_1,\ldots,y_m\}^{\p}$.
\end{theorem}
\begin{proof}
Clearly, $X\subseteq \V(\I(X))$. Let $X=\V(F)$ be a $\p$-subvarieties of $\A_k^m$. Since $F\subset \I(\V(F))$, we have
\[X=\V(F)\supseteq\V(\I(\V(F)))=\V(\I(X)).\]
Therefore, $\V(\I(X))=X$.

Let $I$ be a $\p$-perfect ideal of $k\{y_1,\ldots,y_m\}^{\p}$. Then by Proposition \ref{prop}, $\I(\V(I))=\{I\}^{\p}=I$.
\end{proof}

Suppose $I$ is a $\p$-ideal of $k\{y_1,\ldots,y_m\}^{\p}$. Let $k\{y_1,\ldots,y_m\}^{\p}/I$ be the quotient $\sigma$-ring of $k\{y_1,\ldots,y_m\}^{\p}$ by $I$ and $\pi\colon k\{y_1,\ldots,y_m\}^{\p}\rightarrow k\{y_1,\ldots,y_m\}^{\p}/I$ the quotient map. We define {\em $\p$-perfect ideals} and {\em $\p$-prime ideals} of $k\{y_1,\ldots,y_m\}^{\p}/I$ are $\sigma$-ideals of $k\{y_1,\ldots,y_m\}^{\p}/I$ whose preimages under $\pi$ are $\p$-perfect ideals and $\p$-prime ideals of $k\{y_1,\ldots,y_m\}^{\p}$ respectively.
\begin{cor}
Let $X$ be an affine $\p$-variety. Then the following map
\[Y\mapsto\{f\in k\{X\}\mid f(a)=0,\forall K\in \mathscr{E}_k, \forall a\in Y(K)\}\]
is an inclusion reversing bijection between the set of all $\p$-subvarieties of $X$ and the set of all $\p$-perfect ideals of $k\{X\}$.
\end{cor}
\begin{proof}
It is clear from the definitions and Theorem \ref{apv-thm1}.
\end{proof}

Let $k\{y_1,\ldots,y_m\}^{\p}/I$ be a quotient $\sigma$-ring of $k\{y_1,\ldots,y_m\}^{\p}$. A {\em $\p$-term} and a {\em $\p$-monomial} of $k\{y_1,\ldots,y_m\}^{\p}/I$ are elements which are equal to the images of a $\p$-term and a $\p$-monomial of $k\{y_1,\ldots,y_m\}^{\p}$ under the quotient map respectively.
\begin{definition}
Let $k\{X\}$ and $k\{Y\}$ be two affine $k$-$P[\sigma]$-algebras. A $k$-$\sigma$-algebra homomorphism $\phi\colon k\{X\}\rightarrow k\{Y\}$ is called a {\em morphism of $k$-$P[\sigma]$-algebras} if $\phi$ maps $\p$-terms to $\p$-terms.
\end{definition}
\begin{definition}
Let $X\subseteq \mathbb{A}_k^m$ and $Y\subseteq \mathbb{A}_k^n$ be two affine $k$-$P[\sigma]$-varieties. A morphism of functors $f\colon X\rightarrow Y$ is called a {\em morphism of affine $k$-$P[\sigma]$-varieties} if there exist $\p$-terms $f_1,\ldots,f_n\in k\{y_1,\ldots,y_m\}^{P[\sigma]}$ such that $f(a)=(f_1(a),\ldots,f_n(a))$ for every $K\in \mathscr{E}_k$ and every $a\in X(K)$.
\end{definition}

Let $f\colon X\rightarrow Y$ be a morphism of affine $k$-$P[\sigma]$-varieties. Define
\[\phi\colon k\{z_1,\ldots,z_n\}^{P[\sigma]}\rightarrow k\{X\}, z_i\mapsto \bar{f_i}.\]
Then for $a\in X(K)$ and $g\in k\{z_1,\ldots,z_n\}^{P[\sigma]}$, $\phi(g)(a)=g(f_1(a),\ldots,f_n(a))=g(f(a))$. Because $f(a)\in Y(K)$, we have $\phi(g)(a)=0$ if $g\in \mathbb{I}(Y)$. So $\phi(g)=0$ for $g\in \mathbb{I}(Y)$. This shows that $\phi$ induces a morphism of $k$-$P[\sigma]$-algebras
\[f^*\colon k\{Y\}\rightarrow k\{X\}, \bar{z_i}\mapsto \bar{f_i}.\]
We call $f^*$ the morphism {\em dual} to $f$, and $f^*(g)=g\circ f$, for any $g\in k\{Y\}$.
\begin{theorem}
Let $k$ be a $\sigma$-field. The category of affine $k$-$P[\sigma]$-varieties is antiequivalent to the category of affine $k$-$P[\sigma]$-algebras.
\end{theorem}
\begin{proof}
Let $X\subseteq \mathbb{A}_k^m$ and $Y\subseteq \mathbb{A}_k^n$ be two affine $k$-$P[\sigma]$-varieties. We need to show that
\[\Hom(X,Y)\rightarrow \Hom(k\{Y\},k\{X\}),f\mapsto f^*\]
is bijective. First for the injectivity, let $f,g\in \Hom(X,Y)$ such that $f^*=g^*$. Then $h(f(a))=f^*(h)(a)=g^*(h)(a)=h(g(a))$ for every $h\in k\{z_1,\ldots,z_n\}^{\p}$, $K\in \mathscr{E}_k$ and $a\in X(K)$. Choose $h$ to be the coordinate functions. This shows that $f=g$.

Now we show that the map is surjective. Let $\phi\colon k\{Y\}\rightarrow k\{X\}$ be a morphism of affine $k$-$P[\sigma]$-algebras, where $k\{Y\}=k\{z_1,\ldots,z_n\}^{P[\sigma]}/\mathbb{I}(Y), k\{X\}=k\{y_1,\ldots,y_m\}^{P[\sigma]}/\mathbb{I}(X)$. Suppose that $\phi(\bar{z_i})=\bar{f_i}\in k\{X\},i=1,\ldots,n, f_i$ $\p$-terms. Define $f\colon X\rightarrow \mathbb{A}_k^n$ by $f=(f_1,\ldots,f_n)$. It is easy to check that $f$ is actually mapping into $Y$. so $f\colon X\rightarrow Y$ is a morphism of $k$-$P[\sigma]$-varieties. Clearly $\phi=f^*$.
\end{proof}

Suppose that $R$ is an affine $k$-$P[\sigma]$-algebra. Let $\Spec^{P[\sigma]}(R)$ be the set of all $P[\sigma]$-prime ideals of $R$. Let $F\subseteq R$ and set
$$\mathcal{V}(F):=\{\mathfrak{p}\in \Spec^{P[\sigma]}(R)\mid F\subseteq \mathfrak{p}\}\subseteq \Spec^{P[\sigma]}(R).$$

By Theorem \ref{pi-thm1}, $\mathcal{V}(F)=\mathcal{V}(\{F\}^{P[\sigma]})$. The following lemma is easy to check.
\begin{lemma}\label{apv-lemma2}
Let $R$ be an affine $k$-$P[\sigma]$-algebra and $F,G,F_i\subseteq R$. Then
\begin{enumerate}
\item $\mathcal{V}({0})=\Spec^{P[\sigma]}(R)$ and $\mathcal{V}(R)=\varnothing$;
\item $\mathcal{V}(F)\cup \mathcal{V}(G)=\mathcal{V}(FG)$;
\item $\bigcap_i\mathcal{V}(F_i)=\mathcal{V}(\bigcup_iF_i)$.
\end{enumerate}
\end{lemma}

Lemma \ref{apv-lemma2} shows that $\Spec^{P[\sigma]}(R)$ is a topological space with closed sets of the forms $\mathcal{V}(F)$.

For $f\in R$, set $D^{P[\sigma]}(f):=\Spec^{P[\sigma]}(R)\backslash \mathcal{V}(f)=\{\mathfrak{p}\in \Spec^{P[\sigma]}(R)\mid f\notin \mathfrak{p}\}$. We call $D^{P[\sigma]}(f)$ a {\em basis open subset} of $\Spec^{P[\sigma]}(R)$.
\begin{definition}\label{apv-def}
Let $X$ be an affine $k$-$P[\sigma]$-variety. Then the {\em topological space} of $X$ is $\Spec^{P[\sigma]}(k\{X\})$ equipped with the above topology.
\end{definition}

Similarly to Theorem \ref{pdag-thm}, we have
\begin{theorem}
Let $X$ be an affine $k$-$P[\sigma]$-variety. There is a natural bijection between the set of equivalence classes of solutions of $\mathbb{I}(X)$ and $\Spec^{P[\sigma]}(k\{X\})$.
\end{theorem}

As before, we shall not strictly distinguish between a $P[\sigma]$-variety and its topological space. Namely, we will use $X$ to mean the $P[\sigma]$-variety or its topological space.

\section{Affine Toric $P[\sigma]$-Varieties}
In this section, we will study the properties of affine toric $\p$-varieties. Every affine toric $\p$-variety corresponds to an affine $P[x]$-semimodule. The story is very similar to what we do for affine toric $\sigma$-varieties. But here, more properties will be proved. $k$ is always assumed to be a fixed $\sigma$-field.
\subsection{Affine Toric $\p$-Varieties and Affine $P[x]$-Semimodules}
We start by recalling some basic facts about affine toric $\sigma$-varieties which are proved in \cite{dd-tdv}.

Let $(\A^*)^n$ be the functor from $\mathscr{E}_k$ to $\mathscr{E}_k^n$ satisfying $(\A^*)^n(K)=(K^*)^n$ where $K\in \mathscr{E}_k$ and $K^*=K\backslash \{0\}$. Let $U=\{\bu_1,\ldots,\bu_m\}\subset\Z[x]^n$ and $\T=(t_1,\ldots,t_n)$ an $n$-tuple of $\sigma$-indeterminates. We define the following map
\begin{equation}\label{tvsm-equ1}
\theta\colon(\A^*)^n \longrightarrow \A^m,  \T \mapsto\T^{U} = (\T^{\bu_1},  \ldots, \T^{\bu_m}).
\end{equation}
The functor $T_U^*$ from $\mathscr{E}_k$ to $\mathscr{E}_k^m$ with $T_U^*(K)=\image(\theta_K)$ is a {\em quasi $\sigma$-torus}.

\begin{definition}
Given a finite set $U\subseteq \Z[x]^n$, the {\em affine toric $\sigma$-variety} $X_U$ is defined to be the Cohn closure of the image of the map $\theta$ from (\ref{tvsm-equ1}) in $\A^m$.
%Precisely, let
%\begin{equation}\label{eq-atv0}
%I_U:=\{f\in k\{y_1,\ldots,y_m\}\mid f(\T^{\bu_1}, \ldots, \T^{\bu_m})=0\}.
%\end{equation}
%Then the affine toric $\sigma$-variety defined by $U$ is $X_{U}=\V(I_{U})$.
\end{definition}

%In \cite{dd-tdv}, the following theorem is proved.
%\begin{theorem}\label{astv-thm2}
%An affine $\sigma$-variety $X$ is toric if and only if $\I(X)$ is a toric $\sigma$-ideal.
%\end{theorem}

Given a finite set $U\subseteq \Z[x]^n$, recall that the affine $\mathbb{N}[x]$-semimodule generated by $U$ is $S=\mathbb{N}[x](U)=\{\sum_{i=1}^m g_i\bu_i\mid g_i\in \mathbb{N}[x], 1\le i\le m\}$. For every affine $\mathbb{N}[x]$-semimodule $S$, there is a $k$-$\D$-algebra $k[S]$ with $S$ as a basis, that is
\[k[S]:=\bigoplus_{\bu\in S}k\chi^{\bu}=\{\sum_{\bu\in S}c_{\bu}\chi^{\bu}\mid c_{\bu}\in k \textrm{ and } c_{\bu}=0 \textrm{ for all but finitely many } \bu\}.\]

In \cite{dd-tdv}, it is proved that an affine $\D$-variety $X$ is toric if and only if there exists an affine $\N[x]$-semimodule $S$ such that $X\simeq\Spec^{\D}(k[S])$. Furthermore, the category of affine toric $\sigma$-varieties with toric morphisms is antiequivalent to the category of affine $\N[x]$-semimodules with $\N[x]$-semimodule morphisms.

Each affine toric $\sigma$-variety contains a $\D$-torus as an open subset and extends the group action on itself. In \cite{dd-tdv}, it is proved that an affine $\sigma$-variety $T$ is a $\sigma$-torus if and only if there exists a $\Z[x]$-lattice $M$ such that $T\simeq \Spec^{\sigma}(k[M])$.
%If $X=\Spec^{\sigma}(k[S])$ is an affine toric $\sigma$-variety, then the $\D$-torus of $X$ is $T\simeq \Spec^{\sigma}(k[S^{md}])$, where $S^{md}:=\{\bu-\bv\mid\bu,\bv\in S\}$ is the $\Z[x]$-lattice generated by $S$. %Actually ,we have
%\begin{theorem}\label{th-TT}
%An affine $\sigma$-variety $T$ is a $\sigma$-torus if and only if there exists a $\Z[x]$-lattice $M$ such that $T\simeq \Spec^{\sigma}(k[M])$.
%\end{theorem}

\begin{remark}
Let $T=\Spec^{\sigma}(k[M])$ be a $\sigma$-torus. A {\em character} of $T$ is a morphism of $\sigma$-algebraic groups $\chi\colon T\rightarrow (\A^*)^1$. Denote all characters of $T$ by $X(T)$. Then $X(T)\subseteq k[M]$. Every $\bu\in M$ gives a character $\chi^{\bu}\colon T\rightarrow (\A^*)^1$ which satisfies that for each $K\in\mathscr{E}_k$ and an element $\phi$ of $T(K)$, $\chi^{\bu}(\phi)=\phi(\bu)\in K^*$. Actually, all characters of $T$ arise in this way. Thus $X(T)\simeq M$. In particular, $X((\A^*)^n)\simeq \Z[x]^n$.
\end{remark}

Now we give the definition of affine toric $P[\sigma]$-varieties.
\begin{definition}
An affine $P[\sigma]$-variety over the $\sigma$-field $k$ is said to be {\em toric} if it is the closure of a quasi $\sigma$-torus $T_U^*\subseteq \A^m$ in $\A^m$ under the topology in Definition \ref{apv-def}. More precisely,  let
\begin{equation}\label{eq-atpv0}
J_U:=\{f\in k\{y_1,\ldots,y_m\}^{P[\sigma]}\mid f(\T^{\bu_1},\ldots,\T^{\bu_m})=0\}.
\end{equation}
Then the affine toric $P[\sigma]$-variety defined by $U$ is $X_U=\V(J_U)$.
\end{definition}

\begin{example}
$\A^m$ is an affine toric $\p$-variety with quasi $\D$-torus $(\A^*)^m$.
\end{example}

For $\bv\in\Z[x]^m$, we denote $\bv_+\in P[x]^m$ with $(\bv_+)_i=\bv_i$ if $\bv_i\in P[x]$, and $(\bv_+)_i=0$ otherwise; and denote $\bv_-\in P[x]^m$ with $(\bv_-)_i=-\bv_i$ if $\bv_i\notin P[x]$, and $(\bv_-)_i=0$ otherwise.
\begin{definition}
Given a $\Z[x]$-lattice $L\subseteq \Z[x]^m$, we define a binomial $P[\sigma]$-ideal $J_L\subseteq k\{y_1,\ldots,y_m\}^{\p}$ associated with $L$
$$J_L:=[\Y^{\bv_1}-\Y^{\bv_2}\mid \bv_1-\bv_2\in L, \bv_1,\bv_2\in P[x]^m]=[\Y^{\bv_+}-\Y^{\bv_-}\mid \bv\in L].$$
$L$ is called the {\em support lattice} of $J_L$. If $L$ is a toric $\Z[x]$-lattice, then $J_L$ is called a {\em toric $P[\sigma]$-ideal}.
\end{definition}

\begin{lemma}\label{atpv-lemma1}
Let $X_U$ be the affine toric $\p$-variety defined in (\ref{eq-atpv0}). Then $J_U=\I(X_U)$ is a toric $\p$-ideal whose support lattice is $L=\Syz(U)$.
\end{lemma}
\begin{proof}
By Remark \ref{pi-re1}, $L$ is a toric $\Z[x]$-lattice. Then it suffices to show that $J_U=J_L$, where $J_U$ is defined in (\ref{eq-atpv0}).
For $\bv\in L$, we have $(\Y^{\bv}-1)(\T^U)=(\T^{U})^{\bv}-1=\T^{U\bv}-1=0$.
As a consequence, $(\Y^{\bv_+}-\Y^{\bv_-})(\T^U)=0$ and $\Y^{\bv_+}-\Y^{\bv_-}\in J_U$ with $\bv\in L$. Since $J_L$ is generated by $\Y^{\bv_+}-\Y^{\bv_-}$ for $\bv\in L$, we have $J_L\subseteq J_U$.

To prove the other direction, let us consider the following map
\[\theta\colon k\{y_1,\ldots,y_m\}^{\p}\rightarrow k\{t_1^{\pm1},\ldots,t_n^{\pm1}\}, f\mapsto f(\T^{\bu_1},\ldots,\T^{\bu_m}).\]
Define a grading on $k\{y_1,\ldots,y_m\}^{\p}$ by $\deg(\Y^{\bv})=\overline{U}\bv\in\Z[x]^n$. Then $k\{y_1,\ldots,y_m\}^{\p}$ and $k\{t_1^{\pm1},\ldots,t_n^{\pm1}\}$ are both $\Z[x]^n$-graded and $\theta$ is a homogeneous map of degree $\mathbf{0}$. It follows that the kernel of $\theta$ is homogeneous. So an element of $\ker(\theta)$ of degree $\bu$ can be written as $\sum_{\theta(\bv)=\bu}\alpha_{\bv}\Y^{\bv}$ with $\sum_{\bv}\alpha_{\bv}=0$. Such an element is in $J_L$. Hence $J_U=\ker(\theta)\subseteq J_L$.
\end{proof}

The following lemma shows that the inverse of Lemma \ref{atpv-lemma1} is also valid.
\begin{lemma}\label{atpv-lemma2}
If $I$ is a toric $\p$-ideal in $k\{y_1,\ldots,y_m\}^{\p}$, then $\V(I)$ is an affine toric $\p$-variety.
\end{lemma}
\begin{proof}
Since $I$ is a toric $\p$-ideal, then the $\Z[x]$-lattice corresponding to $I$, denoted by $L$, is toric. Suppose $V=\{\bv_{1}, \ldots, \bv_{n}\}\subset \Z[x]^m$ is a set of generators of $L^C$. Regard $V$ as a matrix with columns $\bv_i$ and let $U=\{\bu_1,\ldots,\bu_m\}\subset \Z[x]^n$ be the set of the row vectors of $V$. Consider the affine toric $\p$-variety  $X_U$ defined by $U$. To prove the lemma, it suffices to show that $X_U=\V(I)$ or $I_U=I$. Since toric $\p$-ideals and toric $\Z[x]$-lattices are in a one-to-one correspondence, we only need to show $\Syz(U)=L$. This is clear since $\Syz(U)=\ker(V^{\tau})=(L^C)^C=L$ by Remark \ref{pi-re2}.
\end{proof}

Combining Lemma \ref{atpv-lemma1} and Lemma \ref{atpv-lemma2}, we have
\begin{theorem}\label{th-tvi}
An affine $P[\sigma]$-variety $X$ over $k$ is toric if and only if $\I(X)$ is a toric $P[\sigma]$-ideal.
\end{theorem}

%\begin{prop}\label{aspv-prop1}
%The $P[\sigma]$-torus $T$ defined above is a toric $P[\sigma]$-variety and has a group structure whose group action is a morphism of $\p$-varieties.
%\end{prop}

%Note that if $T\subseteq \A^m$, the group multiplication of $T$ is just the usual termwise multiplication of $\A^m$, namely, $\forall (x_1,\dots,x_m),(y_1,\ldots,y_m)\in T, (x_1,\dots,x_m)\cdot(y_1,\ldots,y_m)=(x_1y_1,\dots,x_my_m)$.

Now we introduce the concept of affine $P[x]$-semimodules.

$S\subseteq \mathbb{Z}[x]^n$ is called a {\em $P[x]$-semimodule} if it satisfies (i) if $\bu,\bv\in S$, then $\bu+\bv\in S$, and (ii) if $g\in P[x]$ and $\bu\in S$, then $g\bu\in S$. Moreover, if there exists a finite subset $U=\{\bu_1,\ldots,\bu_m\}\subset \mathbb{Z}[x]^n$ such that $S=P[x](U)=\{\sum_{i=1}^m g_i\bu_i\mid g_i\in P[x]\}$, then $S$ is called an {\em affine $P[x]$-semimodule}. If $S$ is an affine $P[x]$-semimodule, let $S^{md}=\{\bu-\bv\mid \bu,\bv\in S\}$ be the $\Z[x]$-lattice generated by $S$, and define $\rank(S)=\rank(S^{md})$. A map $\phi\colon S\rightarrow S'$ between two $P[x]$-semimodules is a {\em $P[x]$-semimodule morphism} if $\phi(\bu+\bv)=\phi(\bu)+\phi(\bv),\phi(g\bu)=g\phi(\bu)$ for all $\bu,\bv\in S,g\in P[x]$.
\begin{remark}
For an affine $P[x]$-semimodule $S=P[x](\{\bu_1,\ldots,\bu_m\})$, if some $\bu_i$ can be generated by the other $\{\bu_j\}_{j\neq i}$, i.e., there exist $g_j\in P[x], j\neq i$ such that $\bu_i=\sum_{j\neq i} g_j\bu_j$, then we can delete $\bu_i$ from $\{\bu_1,\ldots,\bu_m\}$ to generate the same affine $P[x]$-semimodule. In the following, we always assume that there doesn't exist such $\bu_i$ in the generating set of $S$.
\end{remark}

For every affine $P[x]$-semimodule $S$, we define a {\em $P[x]$-semimodule algebra} $k[S]$ which is the vector space over $k$ with $S$ as a basis and multiplication induced by the addition of $S$. More concretely,
\[k[S]:=\bigoplus_{\bu\in S}k\chi^{\bu}=\{\sum_{\bu\in S}c_{\bu}\chi^{\bu}\mid c_{\bu}\in k \textrm{ and } c_{\bu}=0 \textrm{ for all but finitely many } \bu\},\]
with multiplication induced by
\[\chi^{\bu_1}\cdot\chi^{\bu_2}=\chi^{\bu_1+\bu_2}.\]
Make $k[S]$ to be a $k$-$\sigma$-algebra by defining
\[\sigma(\chi^{\bu})=\chi^{x\bu},\textrm{ for } \bu\in S.\]

Suppose $S=P[x](U)=P[x](\{\bu_1,\ldots,\bu_m\})$, then $k[S]=k\{\chi^{\bu_1},\ldots,\chi^{\bu_m}\}^{\p}$. When an embedding $S\rightarrow \Z[x]^n$ is given, it induces an embedding $k[S]\rightarrow k[\Z[x]^n]\simeq k\{t_1^{\pm 1},\ldots,t_n^{\pm 1}\}$ where $\T=\{t_1,\ldots,t_n\}$ is a set of $\sigma$-indeterminates. Therefore, $k[S]$ is a $k$-$\sigma$-subalgebra of $k\{t_1^{\pm 1},\ldots,t_n^{\pm 1}\}$ and it follows that $k[S]$ is a $\sigma$-domain. Also, we can view $k[S]$ as an $S$-graded ring. We will see that $k[S]$ is actually the $P[\sigma]$-coordinate ring of an affine toric $P[\sigma]$-variety.
\begin{theorem}\label{aspv-thm3}
Let $X$ be an affine $k$-$P[\sigma]$-variety. Then $X$ is toric if and only if there exists an
affine $P[x]$-semimodule $S$ such that $X\simeq \Spec^{P[\sigma]}(k[S])$.
Equivalently, the $P[\sigma]$-coordinate ring of $X$ is $k[S]$.
\end{theorem}
\begin{proof}
Suppose $X=X_U$ is the affine toric $\p$-variety defined by $U=\{\bu_1,\ldots,\bu_m\}\subset\Z[x]^n$ and $J_U$ is defined in (\ref{eq-atpv0}). Let $S=P[x](U)=P[x](\{\bu_1,\ldots,\bu_m\})$ be the affine $P[x]$-semimodule generated by $U$. Define the following $k$-$\sigma$-algebra homomorphism
$$\theta\colon k\{y_1,\ldots,y_m\}^{\p}\longrightarrow k[S], y_i\mapsto\chi^{\bu_i}, i=1,\ldots,m.$$
The map $\theta$ is surjective by the definition of $k[S]$. If $f\in\ker(\theta)$, then $f(\chi^{\bu_1},\ldots,\chi^{\bu_i})=0$, which is equivalent to $f\in J_U$. It follows that $\ker(\theta)=J_U$ and $k\{y_1,\ldots,y_m\}^{\p}/{J_U}\simeq k[S]$. Therefore,
$X\simeq\Spec^{\p}(k\{y_1,\ldots,y_m\}^{\p}/J_U)=\Spec^{\p}(k[S])$.

Conversely, if $X\simeq\Spec^{\p}(k[S])$, where $S\subseteq\Z[x]^n$ is an affine $P[x]$-semimodule, and $S=P[x](\{\bu_{1}, \ldots,\bu_{m}\})$ for $\bu_{i}\in S$. Let $X_U$ be the affine toric $\p$-variety defined by $U=\{\bu_{1},\ldots,\bu_{m}\}$. Then as we just proved, the $\p$-coordinate ring of $X$ is isomorphic to $k[S]$. Then $X\simeq X_{U}$.
\end{proof}

Suppose $S$ is an affine $P[x]$-semimodule. For each $K\in \mathscr{E}_k$, a map $\phi\colon S\rightarrow K$ is a {\em morphism} from $S$ to $K$ if $\phi$ satisfies $\phi(\sum_ig_i\bu_i)=\prod_i\phi(\bu_i)^{g_i}$, for $\bu_i\in S$ and $g_i\in P[x]$.
\begin{cor}\label{aspv-prop}
Let $X=\Spec^{P[\sigma]}(k[S])$ be an affine toric $P[\sigma]$-variety. Then there is a one-to-one correspondence between $X(K)$ and $\Hom(S,K)$, for each $K\in \mathscr{E}_k$. Equivalently, $X\simeq \Hom(S,\underline{\mspace{18mu}})$ as functors.
\end{cor}
\begin{proof}
By Lemma \ref{apv-lemma1}, for each $K\in \mathscr{E}_k$, an element of $X(K)$ is given by a $k$-$\sigma$-algebra homomorphism $f\colon k[S]\rightarrow K$. Then $f$ induces a morphism $\bar{f}\colon S\rightarrow K$ such that $\bar{f}(\bu)=f(\chi^{\bu})$ for $\bu\in S$. Conversely, given a morphism $\varphi\colon S\rightarrow K$, $\varphi$ extends to a $k$-$\sigma$-algebra homomorphism $\varphi^*\colon k[S]\rightarrow K$ which proves the one-to-one correspondence.
\end{proof}

In the rest of this paper, we always identity an element of $X(K)$ with a morphism from $S$ to $K$ for each $K\in\mathscr{E}_k$.

We have the following definition for $P[\sigma]$-tori.
\begin{definition}
Let $U=\{\bu_1,\ldots,\bu_m\}\subset\Z[x]^n$. The {\em $P[\sigma]$-torus} $\widetilde{T_U}$ defined by $U$ is the closure of the quasi $\sigma$-torus $T_U^*\subseteq\A^m$ in $(\A^*)^m$ under the topology in Definition \ref{apv-def}.
\end{definition}
\begin{remark}\label{atpv-re}
The above definition could be stated in a more precise way. Note that $(\A^*)^m$ is isomorphic to the affine $\p$-variety defined by $I_0=[y_1z_1-1,\ldots,y_mz_m-1]\subseteq k\{y,z\}^{\p}$ in $(\A)^{2m}$. Furthermore, the map
\begin{equation}\label{eq-TH}
\theta\colon(\A^*)^m\longrightarrow (\A)^{2m}
\end{equation}
defined by $\theta(a_1,\ldots,a_m)=(a_1,\ldots,a_m,a_1^{-1},\ldots,a_m^{-1})$ gives a one-to-one correspondence between $\p$-subvarieties of $(\A^*)^m$ and affine $\p$-varieties contained in $\V(I_0)$. Then the $\p$-torus $\widetilde{T_U}$ is the preimage of the affine $\p$-variety $X_{U\cup(-U)}$ in $\A^{2m}$ under the map $\theta$.
\end{remark}

Let $U=\{\bu_1,\ldots,\bu_m\}\subset\Z[x]^n$. Let $\widetilde{T_U}$ and $X_U$ be the $P[\sigma]$-torus and the affine toric $P[\sigma]$-variety defined by $U$ respectively. Then by definition, $\widetilde{T_U} = X_U\cap (\A^*)^m$ and $X_U=\overline{\widetilde{T_U}}$.

\begin{prop}\label{atpv-prop3}
Let $\widetilde{T}$ be an affine $P[\sigma]$-variety. Then $\widetilde{T}$ is a $P[\sigma]$-torus if and only if there exists a $\Z[x]$-lattice $M$ such that $\widetilde{T}\simeq \Spec^{P[\sigma]}(k[M])$.
\end{prop}
\begin{proof}
Suppose $\widetilde{T}$ is defined by $U$ and let $M=\Z[x](U)$. Since $\widetilde{T}\simeq X_{U\cup(-U)}$(Remark \ref{atpv-re}), we just need to show the $\p$-coordinate ring of $X_{U\cup(-U)}$ is $k[M]$. By definition, $X_{U\cup(-U)}$ is the affine toric $\p$-variety defined by $U\cup(-U)$. Thus by Theorem \ref{aspv-thm3}, the $\p$-coordinate ring of $X_{U\cup(-U)}$ is $k[P[x](U\cup(-U))]=k[M]$.

Conversely, suppose $M=\Z[x](U)$ and $U$ is a finite subset of $\Z[x]^n$. Then by the proof of the above necessity, $U$ defines a $\p$-torus $\widetilde{T_U}$ whose $\p$-coordinate ring is $k[M]$. Since $\widetilde{T}\simeq \widetilde{T_U}$, $\widetilde{T}$ is a $\p$-torus.
\end{proof}

Suppose $S$ is an affine $P[x]$-semimodule and $S^{md}:=\{\bu-\bv\mid\bu,\bv\in S\}$ is the $\Z[x]$-lattice generated by $S$. Let $X=\Spec^{\p}(k[S])$ and $\widetilde{T}=\Spec^{\p}(k[S^{md}])$. As a corollary of Proposition \ref{atpv-prop3}, we have
\begin{cor}\label{aspv-cor2}
For each $K\in \mathscr{E}_k$, there is a one-to-one correspondence between elements of $\widetilde{T}(K)$ and $\Hom(S^{md},K)$. Equivalently, $\widetilde{T}\simeq \Hom(S^{md},\underline{\mspace{18mu}})$ as functors.
\end{cor}
\begin{proof}
For each $K\in \mathscr{E}_k$, suppose $\gamma\colon S\rightarrow K$ is an element of $X(K)$ which lies in $\widetilde{T}(K)$. Since elements of $\widetilde{T}(K)$ are invertible, $\gamma(S)\subseteq K^*$ and hence $\gamma$ can be extended to $S^{md}$, $\gamma\colon S^{md}\rightarrow K^*$. So the one-to-one correspondence follows from Proposition \ref{aspv-prop}.
\end{proof}

Let $T=\Spec^{\sigma}(k[M])$ and $\widetilde{T}=\Spec^{P[\sigma]}(k[M])$ be the $\D$-torus and the $\p$-torus associated with a $\Z[x]$-lattice $M$ respectively. Since by \cite[Proposition 5.7]{dd-tdv}, $T\simeq \Hom(M,\underline{\mspace{18mu}})$ and by Corollary \ref{aspv-cor2}, $\widetilde{T}\simeq \Hom(M,\underline{\mspace{18mu}})$, we know that $T=\widetilde{T}$. So a $\p$-torus is actually a $\sigma$-torus and we will write $T$ for $\widetilde{T}$.

\begin{prop}
Let $U=\{\bu_1,\ldots,\bu_m\}\subset\Z[x]^n$ and $X_U=\Spec^{P[\sigma]}(k[P[x](U)])$ the affine toric $\p$-variety defined by $U$. Then $X_U$ is an irreducible $\p$-variety containing the $\D$-torus $T_U$ as an open subset and is of $\D$-dimension $\rank(U)$.
\end{prop}
\begin{proof}
Let $S=P[x](U)$. Since $X_U$ has the $\p$-coordinate ring $k[S]$ by Theorem \ref{aspv-thm3} and $k[S]$ is a $\D$-domain, $X_U$ is irreducible. The inclusion $i\colon S\hookrightarrow M=S^{md}$ induces a $k$-$\sigma$-algebra homomorphism $i^*\colon k[S]\hookrightarrow k[M]$ which corresponds to a morphism $j\colon T_U\rightarrow X_U$. $S^{md}$ is generated as a $P[x]$-semimodule by $S$ and $-(\bu_1+\ldots+\bu_m)$. This implies that $k[M]$ is the localization of $k[S]$ at the element $\chi^{\bu_1+\ldots+\bu_m}$. Therefore, $j$ embeds $T_U$ into $X_U$ as a principal affine open subset. Since the $\D$-dimension of $T_U$ is $\rank(U)$, the $\D$-dimension of $X_U$ is equal to $\rank(U)$.
\end{proof}

\subsection{$\D$-Algebraic Group Action on $X$}
As in the $\sigma$-case, there is a $\D$-algebraic group action on an affine toric $\p$-variety.
Let $X$ and $T$ be the affine toric $\p$-variety and the $\D$-torus associated with an affine $\p$-semimodule $S$ respectively. We describe how $T$ acts on $X$ as a $\D$-algebraic group.
Define a map $X\times X\rightarrow X\colon ((x_1,\ldots,x_m),(y_1,\ldots,y_m))\mapsto(x_1y_1,\ldots,x_my_m)$, for $(x_1,\ldots,,x_m),(y_1,\ldots,y_m)\in X$. It can be described using $P[x]$-semimodule morphisms as follows: for any $K\in \mathscr{E}_k$, let $\varphi,\psi \colon S\rightarrow K$ be two elements of $X(K)$, then $(\varphi,\psi)\mapsto\varphi\psi\colon S\rightarrow K,\varphi\psi(\bu)=\varphi(\bu)\psi(\bu)$, for $\bu\in S$. This corresponds to the $k$-$\sigma$-algebra homomorphism $\Phi\colon k[S]\rightarrow k[S]\otimes k[S]$ such that $\Phi(\chi^{\bu})=\chi^{\bu}\otimes\chi^{\bu}$, for $\bu\in S$.

Via the embedding $T\subseteq X$, the operation on $X$ induces a map $T\times X\rightarrow X$ which is clearly a $\sigma$-algebraic group action on $X$ and extends the group action of $T$ on itself. It corresponds to the $k$-$\sigma$-algebra homomorphism $k[S]\rightarrow k[S^{md}]\otimes k[S]$, $\chi^{\bu}\mapsto \chi^{\bu}\otimes \chi^{\bu}$, for $\bu\in S$.

The following theorem shows that if an affine $\p$-variety contains a $\D$-torus as an open subset extending the group action of the $\D$-torus on itself, then it is toric. In other words, the theorem gives a description of affine toric $P[\sigma]$-varieties in terms of $\D$-algebraic group actions.
\begin{theorem}\label{atpv-thm4}
Let $X$ be an affine $k$-$P[\sigma]$-variety, $T\subseteq X$ an open subset which is a $\D$-torus such that the group action of $T$ on itself extends to a $\D$-algebraic group action on $X$. Then there is an affine $P[x]$-semimodule $S$ and an isomorphism $X\simeq \Spec^{P[\sigma]}(k[S])$. In other words, $X$ is an affine toric $P[\sigma]$-variety.
\end{theorem}
\begin{proof}
By Proposition \ref{atpv-prop3}, there is a $\mathbb{Z}[x]$-lattice $M$ such that $T\simeq \Spec^{P[\sigma]}(k[M])$. The open embedding $T\subseteq X$ induces $k\{X\}\subseteq k[M]$. Since the action of $T$ on itself extends to a $\D$-algebraic group action on $X$, we have the following commutative diagram:
\begin{equation}\label{atpv-diag2}
\begin{gathered}
\xymatrix{T\times T \ar[r]^(0.6){\phi}\ar[d]&T\ar[d]\\T\times X\ar[r]^(0.6){\tilde{\phi}}&X}
\end{gathered}
\end{equation}
where $\phi$ is the group action of $T$, $\tilde{\phi}$ is the extension of $\phi$ to $T\times X$.

From (\ref{atpv-diag2}), we obtain the following commutative diagram of the corresponding $\p$-coordinate rings:
$$\begin{gathered}
\xymatrix{k\{X\}\ar[r]^(0.35){\tilde{\Phi}}\ar[d]&k[M]\otimes_k k\{X\}\ar[d]\\k[M]\ar[r]^(0.35){\Phi}&k[M]\otimes_k k[M]}
\end{gathered}$$
where the vertical maps are inclusions, and $\Phi(\chi^{\bu})=\chi^{\bu}\otimes \chi^{\bu}$ for $\bu\in M$. It follows that if $\sum_{\bu\in M}\alpha_{\bu}\chi^{\bu}$ with finitely many $\alpha_{\bu}\neq 0$ is in $k\{X\}$, then $\sum_{\bu\in M}\alpha_{\bu}\chi^{\bu}\otimes \chi^{\bu}$ is in $k[M]\otimes_k k\{X\}$, so $\alpha_{\bu}\chi^{\bu}\in k\{X\}$ for every $\bu\in M$. This shows that there is a subset $S$ of $M$ such that $k\{X\}=\bigoplus_{\bu\in S}k\chi^{\bu}$. Since $k\{X\}$ is an affine $k$-$P[\sigma]$-subalgebra of $k[M]$, it follows that $S$ is an affine $P[x]$-semimodule. So by Theorem \ref{aspv-thm3}, $X$ is an affine toric $\p$-variety.
\end{proof}

\subsection{Toric Morphisms between Affine Toric $P[\sigma]$-Varieties}
Note that if $\phi\colon S_1\rightarrow S_2$ is a morphism between affine $P[x]$-semimodules, we have an induced $k$-$\sigma$-algebra homomorphism $\bar{\phi}\colon k[S_1]\rightarrow k[S_2]$ such that $\bar{\phi}(\chi^{\bu})=\chi^{\phi(\bu)}$, for all $\bu\in S$, which gives a morphism between affine toric $\p$-varieties $\phi^*\colon\Spec^{\p}(k[S_2])\rightarrow\Spec^{\p}(k[S_1])$. In this subsection, we will show that actually all toric morphisms between affine toric $\p$-varieties arise in this way. First we give the definition of toric morphisms.
\begin{definition}
Let $X_i=\Spec^{\p}(k[S_i])$ be the affine toric $\p$-varieties coming from affine $P[x]$-semimodules $S_i, i=1,2$ with $\D$-tori $T_i$ respectively. A morphism $\phi\colon X_1\rightarrow X_2$ is said to be {\em toric} if $\phi(T_1)\subseteq T_2$ and $\phi|_{T_1}$ is a $\sigma$-algebraic group homomorphism.
\end{definition}

\begin{prop}
Let $\phi\colon X_1\rightarrow X_2$ be a toric morphism of affine toric $\p$-varieties. Then $\phi$ preserves group actions, namely,
\[\phi(t\cdot p)=\phi(t)\cdot \phi(p)\]
for all $t\in T_1$ and $p\in X_1$.
\end{prop}
\begin{proof}
Suppose the action of $T_i$ on $X_i$ is given by the morphism $\varphi_i\colon T_i\times X_i\rightarrow X_i, i=1,2$. To show $\phi$ preserves group actions is equivalent to showing that the following diagram is commutative:
\begin{equation}
\begin{gathered}
\xymatrix{T_1\times X_1\ar[r]^(0.6){\varphi_1}\ar[d]_{\phi|_{T_1}\times\phi}&X_1\ar[d]^{\phi}\\T_2\times X_2\ar[r]^(0.6){\varphi_2}&X_2}
\end{gathered}
\end{equation}
If we replace $X_i$ by $T_i$ in the diagram, then it certainly commutes since $\phi|_{T_1}$ is a $\sigma$-algebraic group homomorphism. And since $T_1\times T_1$ is dense in $T_1\times X_1$, the whole diagram is commutative.
\end{proof}

The following lemma is taken from \cite[Lemma 6.3]{dd-tdv}.
\begin{lemma}\label{tmtv-prop}
Let $T_i$ be the $\sigma$-tori associated with the $\Z[x]$-lattices $M_i,i=1,2$ respectively. Then a map $\phi\colon T_1\rightarrow T_2$ is a $\sigma$-algebraic group homomorphism if and only if the corresponding map of the $\sigma$-coordinate rings $\phi^*\colon k[M_2]\rightarrow k[M_1]$ is induced by a $\Z[x]$-module homomorphism $\hat{\phi}\colon M_2\rightarrow M_1$.
\end{lemma}

\begin{theorem}\label{tmpv-thm}
Let $X_i=\Spec^{\p}(k[S_i])$ be affine toric $\p$-varieties coming from affine $P[x]$-semimodules $S_i, i=1,2$ with $\D$-tori $T_i$ respectively. Then a morphism $\phi\colon X_1\rightarrow X_2$ is toric if and only if it is induced by a $P[x]$-semimodule morphism $\hat{\phi}\colon S_2\rightarrow S_1$.
\end{theorem}
\begin{proof}
``$\Leftarrow$''. Suppose $\hat{\phi}\colon S_2\rightarrow S_1$ is a $P[x]$-semimodule morphism.
Then $\hat{\phi}$ extends to a $\Z[x]$-module homomorphism $\hat{\phi}\colon M_2\rightarrow M_1$, where $M_1=S_1^{md}, M_2=S_2^{md}$.
By Lemma \ref{tmtv-prop}, it induces a morphism of $\sigma$-algebraic groups $\phi\colon T_1\rightarrow T_2$. So $\phi$ is toric.

``$\Rightarrow$''. $\phi$ induces $\phi^*\colon k[S_2]\rightarrow k[S_1]$. Since $\phi$ is toric, $\phi|_{T_1}$ is a $\sigma$-algebraic group homomorphism. By Lemma \ref{tmtv-prop}, it is induced by a $\Z[x]$-module homomorphism $\tilde{\phi}\colon M_2\rightarrow M_1$. This, combined with $\phi^*(k[S_2])\subseteq k[S_1]$, implies that $\tilde{\phi}$ induces a $P[x]$-semimodule morphism $\hat{\phi}\colon S_2\rightarrow S_1$.
\end{proof}

Combining Theorem \ref{aspv-thm3} with Theorem \ref{tmpv-thm}, we have
\begin{theorem}
The category of affine toric $P[\sigma]$-varieties with toric morphisms is antiequivalent to the category of affine $P[x]$-semimodules with $P[x]$-semimodule morphisms.
\end{theorem}

\subsection{$T$-Orbits of Affine Toric $P[\sigma]$-Varieties}
In this subsection, we will establish a one-to-one correspondence between the irreducible $T$-invariant $\p$-subvarieties of an affine toric $\p$-variety and the faces of the corresponding affine $P[x]$-semimodule. Also, a one-to-one correspondence between $T$-orbits and faces of corresponding affine $P[x]$-semimodules is given for a class of affine $P[x]$-semimodules.
\begin{definition}
Let $S$ be an affine $P[x]$-semimodule. Define a {\em face} of $S$ to be a $P[x]$-subsemimodule $F\subseteq S$ such that
\begin{description}
\item[(1)] for $\bu_1,\bu_2\in S$, $\bu_1+\bu_2\in F$ implies $\bu_1,\bu_2\in F$;
\item[(2)] for $g\in P[x]^*$ and $\bu\in S$, $g\bu\in F$ implies $\bu\in F$,
\end{description}
which is denoted by  $F\preceq S$.
\end{definition}

Note that if $S=P[x](\{\bu_1,\bu_2,\ldots,\bu_m\})$, and $F$ is a face of $S$, then $F$ is generated by a subset of $\{\bu_1,\bu_2,\ldots,\bu_m\}$ as a $P[x]$-semimodule. It follows that $F$ is an affine $P[x]$-semimodule and $S$ has only finitely many faces. $S$ is a face of itself. A {\em proper face} of $S$ is a face strictly contained in $S$.
It is easy to check that the intersection of two faces is again a face and a face of a face is again a face. A face of rank $1$ is called an {\em edge}, and a face of rank $\rank(S)-1$ is called a {\em facet}. Note that unlike affine $\N[x]$-semimodules, a proper face of an affine $P[x]$-semimodule $S$ must have rank less than $\rank(S)$. $S$ is said to be {\em pointed} if $S\cap (-S)=\{\mathbf{0}\}$, i.e., $\{\mathbf{0}\}$ is a face of $S$.

\begin{example}
Let $S=P[x](\{x-1,x-2\})$. Then $S$ has two faces: $F_1=\{0\}$ and $F_2=S$.
\end{example}
\begin{example}
Let $S=P[x](\{\bu_1=(x,1),\bu_2=(x,2),\bu_3=(x,3)\})$. Then $S$ has four faces: $F_1=\{\mathbf{0}\}$, $F_2=P[x](\{\bu_1\})$, $F_3=P[x](\{\bu_3\})$ and $F_4=S$.
\end{example}
\begin{example}
Let $S=P[x](\{\bu_1=(x,1,1),\bu_2=(1,x,1),\bu_3=(1,1,x),\bu_4=(1,1,1)\})$. Then $S$ has eight faces: $F_1=\{\mathbf{0}\}$, $F_2=P[x](\{\bu_1\})$, $F_3=P[x](\{\bu_2\})$, $F_4=P[x](\{\bu_3\})$, $F_5=P[x](\{\bu_2,\bu_3\})$, $F_6=P[x](\{\bu_1,\bu_3\})$, $F_7=P[x](\{\bu_1,\bu_2\})$ and $F_8=S$.
\end{example}

\begin{lemma}\label{oapv-lemma}
A subset $F$ of $S$ is a face if and only if $k[S\backslash F]$ is a $P[\sigma]$-prime ideal of $k[S]$.
\end{lemma}
\begin{proof}
Let $I=k[S\backslash F]:=\bigoplus_{u\in S\backslash F}k\chi^{\bu}$.

``$\Rightarrow$''. $F$ can be viewed as a face of $S$ as an $\N[x]$-semimodule, so by \cite[Lemma 6.8]{dd-tdv}, $I$ is a $\sigma$-prime ideal. We need to show $I$ is a $\p$-ideal. Suppose $\Y^{\bu}f\in I$ and $f=\sum_i\alpha_i\Y^{\bu_i}$, then $\Y^{\bu}f=\sum_i\alpha_i\Y^{\bu+\bu_i}$. Since $I$ is $(S\backslash F)$-graded, $\bu+\bu_i\in S\backslash F$. For $g\in P[x]^*$, $\Y^{g\bu}f\in I$ is equivalent to $g\bu+\bu_i\in S\backslash F$ for all $i$. Suppose the contrary, $g\bu+\bu_i\in F$ for some $i$. Because $F$ is a face of $S$, it follows that $\bu,\bu_i\in F$ and hence $\bu+\bu_i\in F$, which is a contradictory.

``$\Leftarrow$''. Since $I$ is a $P[\sigma]$-ideal, $\bu_1\in S\backslash F$ or $\bu_2\in S\backslash F$ implies $\bu_1+\bu_2\in S\backslash F$, and for $g\in P[x]^*$, $\bu\in S\backslash F$ implies $g\bu\in S\backslash F$. As a consequence, $\bu_1+\bu_2\in F$ implies $\bu_1,\bu_2\in F$, and for $g\in P[x]^*$, $g\bu\in F$ implies $\bu\in F$. Moreover, since $I$ is prime, $\bu_1+\bu_2\in S\backslash F$ implies $\bu_1\in S\backslash F$ or $\bu_2\in S\backslash F$. As a consequence, $\bu_1,\bu_2\in F$ implies $\bu_1+\bu_2\in F$. Since $I$ is perfect, for $g\in P[x]^*$, $g\bu\in S\backslash F$ implies $\bu\in S\backslash F$. As a consequence, for $g\in P[x]^*$, $\bu\in F$ implies $g\bu\in F$. Thus $F$ is a face of $S$.
\end{proof}

Let $X=\Spec^{\p}(k[S])$ be an affine toric $\p$-variety and $T$ the $\sigma$-torus of $X$.
A $\p$-subvariety $Y$ of $X$ is said to be {\em invariant} under the action of $T$
if $T\cdot Y\subseteq Y$.
%Let $\Lambda:=\{F\preceq S\mid \textrm{ there exists a morphism } \varphi\colon S\rightarrow K \textrm{ for some } K\in\mathscr{E}_k \textrm{ such that } F=\varphi^{-1}(K^*)\}.$
For a face $F$ of $S$, let $Y=\Spec^{\p}(k[F])$. Without loss of generality, assume that $S=\N[x](\{\bu_1,\ldots,\bu_m\})$ and $F=\N[x](\{\bu_1,\ldots,\bu_r\})$. We always view $Y$ as a $\p$-subvariety of $X$ through the embedding $j\colon Y\rightarrow X, \gamma\in Y(K)\mapsto (\gamma(\bu_1),\ldots,\gamma(\bu_r),0,\ldots,0)\in X(K)$ for each $K\in \mathscr{E}_k$. The following theorem gives a description for irreducible invariant $\p$-subvarieties of $X$ in terms of the faces of $S$.

\begin{theorem}\label{oapv-thm1}
Let $X=\Spec^{P[\sigma]}(k[S])$ be an affine toric $P[\sigma]$-variety and $T$ the $\D$-torus of $X$. Then the irreducible invariant $P[\sigma]$-subvarieties of $X$ under the action of $T$ are in an inclusion-preserving bijection with the faces of $S$. More precisely, if we denote the irreducible invariant $P[\sigma]$-subvariety corresponding to the face $F$ by $D(F)$, then $D(F)$ is defined by the $P[\sigma]$-ideal $k[S\backslash F]=\bigoplus_{u\in S\backslash F}k\chi^{\bu}$ and the $P[\sigma]$-coordinate ring of $D(F)$ is $k[F]=\bigoplus_{u\in F}k\chi^{\bu}$.
\end{theorem}
\begin{proof}
For a face $F$ of $S$, let $Y=\Spec^{\p}(k[F])$. It is clear that $Y$ is invariant under the action of $T$. The defining ideal of $Y$ is $I=k[S\backslash F]$. Hence by Lemma \ref{oapv-lemma}, $Y$ is irreducible.

On the other hand, suppose $Y$ is an irreducible invariant $P[\sigma]$-subvariety of $X$ and is defined by the $P[\sigma]$-ideal $I$. Then $k\{Y\}=k[S]/I$. By definition, $Y$ is invariant under the $\D$-torus action if and only if the action of $T$ on $X$ induces an action on $Y$, i.e., we have the following commutative diagram:
$$\begin{gathered}
\xymatrix{k[S]\ar[r]^(0.37){\phi}\ar[d]&k[M]\otimes k[S]\ar[d]\\k\{Y\}\ar[r]&k[M]\otimes k\{Y\}}
\end{gathered}$$
where $M=S^{md}$. Since $k[M]\otimes k\{Y\}=k[M]\otimes (k[S]/I)\simeq k[M]\otimes k[S]/k[M]\otimes I$, we must have $\phi(I)\subseteq k[M]\otimes I$. As in the proof of Theorem \ref{atpv-thm4}, this is equivalent to the fact that $I$ is an $M$-graded ideal of $k[S]$, i.e., we can write $I=\oplus_{u\in S'}k\chi^{\bu}$, where $S'$ is a subset of $S$. Since $I$ is a $P[\sigma]$-prime ideal, by Lemma \ref{oapv-lemma}, $F:=S\backslash S'$ is a face of $S$. Moreover, since $I=k[S\backslash F]$, $k\{Y\}=k[S]/I=k[F]$.
\end{proof}

\begin{remark}
Suppose $F$ is a face of $S$. Note that for $K\in \mathscr{E}_k$, an element $\gamma\colon S\rightarrow K$ of $X(K)$ lies in $D(F)(K)$ if and only if $\gamma(S\backslash F)=0$.
\end{remark}

Suppose $X$ is an affine toric $\p$-variety with $\D$-torus $T$. By Theorem \ref{atpv-thm4}, for each $K\in \mathscr{E}_k$, $T(K)$ has a group action on $X(K)$, so we have orbits of $T(K)$ in $X(K)$ under the action. To construct a correspondence between orbits and faces, we need a new kind of affine $P[x]$-semimodules. Suppose $S$ is an affine $P[x]$-semimodule, we say $S$ is {\em face-saturated} if for any face $F$ of $S$, a morphism $\varphi\colon F\rightarrow K^*$ can be extended to a morphism $\widetilde{\varphi}\colon S\rightarrow K^*$ for any $K\in \mathscr{E}_k$. A necessary condition for $S$ to be face-saturated is that for any face $F$ of $S$, $F^{md}$ is $P[x]$-saturated in $S^{md}$, that is, for all $g\in P[x]^*$ and $\bu\in S^{md}$, $g\bu\in F^{md}$ implies $\bu\in F^{md}$.

\begin{example}
Let $S=P[x](\{(2,0),(1,1),(0,1)\})$ and $F=P[x](\{(2,0)\})$ a face of $S$. $(1,0)\in S^{md}$. Since $(1,0)\notin F$ and $2(1,0)\in F$, $S$ is not face-saturated.
\end{example}

We also have the following Orbit-Face correspondence theorem.
\begin{theorem}\label{oapv-thm}
Suppose $S$ is a face-saturated affine $P[x]$-semimodule. Let $X=\Spec^{\p}(k[S])$ be the affine toric $\p$-variety associated with $S$ and $T$ the $\sigma$-torus of $X$. Then for each $K\in \mathscr{E}_k$, there is a one-to-one correspondence between the orbits of $T(K)$ in $X(K)$ and the faces of $S$.
\end{theorem}
\begin{proof}
The proof is similar to \cite[Theorem 6.11]{dd-tdv}.
\end{proof}

\section{Projective Toric $P[\sigma]$-Varieties}
In this section, we will define projective toric $P[\sigma]$-varieties.
\subsection{Projective $P[\sigma]$-Varieties}
Let $k$ be a $\sigma$-field. A {\em $\sigma$-projective ($m$-)space} over $k$ is a functor $\PP^m=(\A^{m+1}\backslash\{\mathbf{0}\})/\A^*$ from the category of $\sigma$-field extensions of $k$ to the category of sets given by $\PP^m(K)=(K^{m+1}\backslash\{\mathbf{0}\})/K^*$ for $K\in \mathscr{E}_k$, where $K^*$ acts via homotheties, i.e.\ $\lambda\cdot(a_0,\ldots,a_m)=(\lambda a_0,\ldots,\lambda a_m)$ for $\lambda \in K^*$ and $(a_0,\ldots,a_m)\in K^{m+1}$.
\begin{definition}
A $P[\sigma]$-polynomial $f\in k\{y_0,\ldots,y_m\}^{P[\sigma]}$ is called {\em transformally homogeneous} if for a new $\D$-indeterminate $\lambda$, there exists a $\p$-monomial $M(\lambda)$ in $\lambda$ such that $f(\lambda y_0,\ldots,\lambda y_m)=M(\lambda)f(y_0,\ldots,y_m)$. A $\p$-ideal is {\em homogeneous} if it can be generated by a set of transformally homogeneous $\p$-polynomials.
\end{definition}

\begin{definition}
Suppose that $F$ is a set of transformally homogeneous $P[\sigma]$-polynomials in $k\{y_0,\ldots,y_m\}^{P[\sigma]}$. The {\em projective $P[\sigma]$-variety} over $k$ defined by $F$ is a subfunctor of $\PP^m$ given by $\V_K(F)=\{a\in \PP^m(K)\mid f(a)=0, \forall f\in F\}$ for each $K\in \mathscr{E}_k$.
\end{definition}

If $X$ is a projective $P[\sigma]$-variety, then the ideal $\I(X)$ generated by all transformally homogeneous $\p$-polynomials vanishing on $X$
is called the {\em vanishing ideal} of $X$ and
\[k\{X\}:=k\{y_0,\ldots,y_m\}^{P[\sigma]}/\mathbb{I}(X)\]
is called the {\em homogeneous $P[\sigma]$-coordinate ring} of $X$.

Suppose $X$ is a projective $k$-$P[\sigma]$-variety. Let $\Proj^{P[\sigma]}(k\{X\})$ be the set of all homogeneous $P[\sigma]$-prime ideals of $k\{X\}$ except $\{\bar{y_0},\ldots,\bar{y_m}\}^{P[\sigma]}$. Let $F\subseteq k\{X\}$. We set
$$\mathcal{V}(F):=\{\mathfrak{p}\in \Proj^{P[\sigma]}(k\{X\})\mid F\subseteq \mathfrak{p}\}\subseteq \Proj^{P[\sigma]}(k\{X\}).$$

The following lemma is easy to check.
\begin{lemma}\label{ptv-lemma}
Let $X$ be a projective $k$-$P[\sigma]$-variety and $F,G,F_i\subseteq k\{X\}$. Then
\begin{enumerate}
\item $\mathcal{V}({0})=\Proj^{P[\sigma]}(k\{X\})$ and $\mathcal{V}(k\{X\})=\varnothing$;
\item $\mathcal{V}(F)\cup \mathcal{V}(G)=\mathcal{V}(FG)$;
\item $\bigcap_i\mathcal{V}(F_i)=\mathcal{V}(\bigcup_iF_i)$.
\end{enumerate}
\end{lemma}

Lemma \ref{ptv-lemma} shows that $\Proj^{P[\sigma]}(k\{X\})$ is a topological space with closed sets of the forms $\mathcal{V}(F)$.
\begin{definition}\label{ptv-def}
Let $X$ be a projective $k$-$P[\sigma]$-variety. Then the {\em topological space} of $X$ is $\Proj^{P[\sigma]}(k\{X\})$ equipped with the above topology.
\end{definition}

\subsection{$\Z[x]$-Lattice Points and Projective Toric $P[\sigma]$-Varieties}
Let $k$ be a $\sigma$-field. Note that $\PP^{m-1}$ is a toric $P[\sigma]$-variety with $\D$-torus
\begin{align*}
T_{\PP^{m-1}}&=\PP^{m-1}\backslash\V(y_0\ldots y_{m-1})=\{(a_0,\ldots,a_{m-1})\in \PP^{m-1}\mid a_0\ldots a_{m-1}\neq 0\}\\
&=\{(1,t_1,\dots,t_{m-1})\in \PP^{m-1}\mid t_1,\ldots,t_{m-1}\in \A^*\}\simeq(\A^*)^{m-1}.
\end{align*}
The action of $T_{\PP^{m-1}}$ on itself clearly extends to an action on $\PP^{m-1}$, making $\PP^{m-1}$ a toric $P[\sigma]$-variety. To describe the character lattice of $T_{\PP^{m-1}}$, consider the exact sequence of $\D$-tori
\[1\longrightarrow\A^*\longrightarrow(\A^*)^m\stackrel{\pi}{\longrightarrow}T_{\PP^{m-1}}\longrightarrow1.\]
Thus the character lattice of $T_{\PP^{m-1}}$ is
\[\mathscr{M}_{m-1}:=\{(a_0,\ldots,a_{m-1})\in \Z[x]^{m}\mid \sum^{m-1}_{i=0}a_i=0\}.\]

Let $U=\{\bu_1,\ldots,\bu_m\}\subset \Z[x]^n$ and $\T=(t_1,\ldots,t_n)$ a set of $\sigma$-indeterminates. In Section 5, we have defined the affine toric $P[\sigma]$-variety associated with $U$ as the closure of the image of the following map
\[\theta\colon(\A^*)^n \longrightarrow \A^m,  \T \mapsto\T^U = (\T^{\bu_1},  \ldots, \T^{\bu_m}).\]
To get a projective toric $P[\sigma]$-variety, we regard $\theta$ as a map to $(\A^*)^m$ and compose with the homomorphism $\pi\colon (\A^*)^m\rightarrow T_{\PP^{m-1}}$ to obtain
\begin{equation}\label{proj-1}
(\A^*)^n\stackrel{\theta}{\longrightarrow}(\A)^m\stackrel{\pi}{\longrightarrow}T_{\PP^{m-1}}\subseteq \PP^{m-1}.
\end{equation}
\begin{definition}\label{tpv-def1}
Given a finite set $U\subset \Z[x]^n$, the {\em projective toric $P[\sigma]$-variety} $Y_U$ is the closure in $\PP^{m-1}$ of the image of the map $\pi\circ\theta$ from (\ref{proj-1}) under the topology in Definition \ref{ptv-def}.
\end{definition}

For a finite set $U=\{\bu_1,\ldots,\bu_m\}\subset \Z[x]^n$, we have defined an affine toric $\p$-variety $X_U$ in Section 5 and a projective toric $\p$-variety $Y_U$ in Definition \ref{tpv-def1}. The following proposition reveals the relationship between $X_U$ and $Y_U$. Let $M=\Z[x](U)$ and $L=\Syz(U)$, then we have an exact sequence
\begin{equation}\label{ptpv-equ}
0\longrightarrow L\longrightarrow\Z[x]^m\longrightarrow M\longrightarrow 0.
\end{equation}
The vanishing ideal of $X_U$ is the binomial $\p$-ideal
\[J_L=[\Y^{\bu}-\Y^{\bv}\mid \bu,\bv\in P[x]^m \textrm{ with } \bu-\bv\in L].\]
\begin{prop}\label{ptpv-prop}
For a finite set $U\subset \Z[x]^n$, the followings are equivalent:
\begin{enumerate}
\item[(a)] $\I(X_U)=J_L=\I(Y_U)$;
\item[(b)] $J_L$ is homogeneous;
\item[(c)] There exists a vector $\bv\in\Z[x]^n$ and $g\in\Z[x]$ such that $\langle\bu_i,\bv\rangle=g$ for all $\bu_i\in U$.
\end{enumerate}
\end{prop}
\begin{proof}
(a)$\Leftrightarrow$(b) is easy from the definitions.

(b)$\Rightarrow$(c). Assume $J_L$ is homogeneous and take $\Y^{\bu}-\Y^{\bv}\in J_L$ for $\bu,\bv\in P[x]^m$ and $\bu-\bv\in L$. If $\sum_{i=1}^mu_i\ne\sum_{i=1}^mv_i$, then $\Y^{\bu},\Y^{\bv}\in J_L$ which is impossible. So $\sum_{i=1}^mu_i=\sum_{i=1}^mv_i$. It follows $\bu\cdot(1,\ldots,1)=0$ for all $\bu\in L$. Now apply the functor $\Hom_{\Z[x]}(\underline{\mspace{18mu}},\Z[x])$ to (\ref{ptpv-equ}) and we obtain an exact sequence
\begin{equation}\label{ptv-eq}
N\longrightarrow\Z[x]^m\longrightarrow\Hom_{\Z[x]}(L,\Z[x])\longrightarrow0
\end{equation}
where $N:=\Hom_{\Z[x]}(M,\Z[x])$. The above argument shows that $(1,\ldots,1)\in\Z[x]^m$ maps to zero in $\Hom_{\Z[x]}(L,\Z[x])$ and hence there exists $\varphi\in N$ such that $\varphi(\bu_i)=1$ for all $i$. By Lemma \ref{atv-lemma}, there exists a vector $\bv\in\Z[x]^n$ and $g\in\Z[x]$ such that $g\varphi=\varphi_{\bv}=\langle\underline{\mspace{18mu}},\bv\rangle$. In particular, $\langle\bu_i,\bv\rangle=g$ for all $\bu_i\in U$.

(c)$\Rightarrow$(b). The vector $\bv\in\Z[x]^n$ gives $\varphi_{\bv}=\langle\underline{\mspace{18mu}},\bv\rangle\in N$ such that $\varphi_{\bv}(\bu_i)=g$ for all $i$. From (\ref{ptv-eq}), $(g,\ldots,g)$ maps to zero in $\Hom_{\Z[x]}(L,\Z[x])$. It follows that for any $\bu\in L$, $\sum_{i=1}^mgu_i=0$ and hence $\sum_{i=1}^mu_i=0$. So $J_L$ is homogeneous.
\end{proof}

Given $U=\{\bu_1,\ldots,\bu_m\}\subset \Z[x]^n$, we set
\[\Z[x]'(U)=\{\sum_{i=1}^ma_i\bu_i\mid a_i\in\Z[x],\sum_{i=1}^ma_i=0\}.\]
\begin{prop}
Let $Y_U$ be the projective toric $\p$-variety defined by $U$. Then:
\begin{enumerate}
\item[(a)] The character lattice of the $\D$-torus of $Y_U$ is $\Z[x]'(U)$.
\item[(b)] The $\sigma$-dimension of $Y_U$ is the dimension of the smallest affine subspace of $\Q(x)^m$ containing $U$. More concretely,
\begin{equation*}
\textrm{$\sigma$-dim}(Y_U)=\begin{cases}
\rank(U)-1, \textrm{ if $U$ satisfies the conditions of Proposition \ref{ptpv-prop}};\\
\rank(U), \textrm{ otherwise}.
\end{cases}
\end{equation*}
\end{enumerate}
\end{prop}
\begin{proof}
(a) The proof is similar to Proposition 2.1.6(a) in \cite[p.58]{cox-2010}.\\
(b) By (a), the $\sigma$-dimension of $Y_U$ equals to the rank of $\Z[x]'(U)$. Let $U'=\{\bu_2-\bu_1,\ldots,\bu_m-\bu_1\}$. It is easy to check that $\Z[x]'(U)=\Z[x](U')$. So $\rank(\Z[x]'(U))=\rank(U')$ and the conclusions of (b) then follow.
\end{proof}

In the following we show that a projective toric $P[\sigma]$-variety is actually covered by a series of affine toric $P[\sigma]$-varieties. Let $O_i=\PP^{m-1}\backslash \V(y_i)$ which is an affine open subset containing the $\D$-torus $T_{\PP^{m-1}}$. We have
\[T_{Y_U}=Y_U\cap T_{\PP^{m-1}}\subseteq Y_U\cap O_i.\]
Since $Y_U$ is the closure of $T_{Y_U}$ in $\PP^{m-1}$, it follows that $Y_U\cap O_i$ is the closure of $T_{Y_U}$ in $O_i\simeq \A^{m-1}$. Thus $Y_U\cap O_i$ is an affine toric $P[\sigma]$-variety. We will determinate the affine $P[x]$-semimodule associated with $Y_U\cap O_i$.
$O_i\simeq \A^{m-1}$ is given by
\[(a_1,\ldots,a_m)\mapsto (a_1/a_i,\ldots,a_{i-1}/a_i,a_{i+1}/a_i,\ldots,a_m/a_i).\]
Combining this with the map (\ref{proj-1}), we see that $Y_U\cap O_i$ is the closure of the image of the map $(\A^*)^m\rightarrow \A^{m-1}$ given by
\begin{equation}\label{proj-2}
\T\mapsto (\T^{\bu_1-\bu_i},\ldots,\T^{
\bu_{i-1}-\bu_i},\T^{\bu_{i+1}-\bu_i},\ldots,\T^{\bu_m-\bu_i}).
\end{equation}
If we set $U_i=U-\bu_i=\{\bu_j-\bu_i\mid j\neq i\}$ and $S_i=P[x](U_i)$, it follows that
\[Y_U\cap O_i=X_{U_i}=\Spec^{P[\sigma]}(k[S_i]).\]
So we have the following proposition:
\begin{prop}
Let $Y_U\subseteq \PP^{m-1}$ for $U=\{\bu_1,\ldots,\bu_m\}\subset \Z[x]^n$. Then the affine piece $Y_U\cap O_i$ is the affine toric $P[\sigma]$-variety
\[Y_U\cap O_i=X_{U_i}=\Spec^{P[\sigma]}(k[S_i]),\]
where $U_i=U-\bu_i=\{\bu_j-\bu_i\mid j\neq i\}$ and $S_i=P[x](U_i)$, $i=1,\ldots,m$.
\end{prop}

Besides describing the affine pieces $Y_U\cap O_i$ of $Y_U\subseteq \PP^{m-1}$, we can also describe how they patch together. When $i\neq j$, $Y_U\cap O_i\cap O_j$ consists of all points of $Y_U\cap O_i$ where $y_j/y_i\neq 0$. By (\ref{proj-2}), this means those points where $\chi^{\bu_j-\bu_i}\neq 0$. Thus
\begin{align*}
Y_U\cap O_i\cap O_j&=\Spec^{P[\sigma]}(k[S_i])_{\chi^{\bu_j-\bu_i}}=\Spec^{P[\sigma]}(k[S_i]_{\chi^{\bu_j-\bu_i}})\\
&=\Spec^{P[\sigma]}(k[S_i+\Z[x](\bu_i-\bu_j)])\subseteq Y_U\cap O_i.
\end{align*}
Also,
\begin{align*}
Y_U\cap O_i\cap O_j&=\Spec^{P[\sigma]}(k[S_j])_{\chi^{\bu_i-\bu_j}}=\Spec^{P[\sigma]}(k[S_j]_{\chi^{\bu_i-\bu_j}})\\
&=\Spec^{P[\sigma]}(k[S_j+\Z[x](\bu_j-\bu_i)])\subseteq Y_U\cap O_j.
\end{align*}
\begin{remark}
One can check that $S_i+\Z[x](\bu_i-\bu_j)=S_j+\Z[x](\bu_j-\bu_i)$.
\end{remark}

\section{Abstract Toric $P[\sigma]$-Varieties}
In this section, we will define abstract toric $P[\sigma]$-varieties through gluing affine toric $P[\sigma]$-varieties along open subsets and generalize the irreducible invariant $\p$-subvarieties-faces correspondence to abstract toric $P[\sigma]$-varieties.
\subsection{Gluing Together Affine $P[\sigma]$-Varieties}
Suppose that we have a finite collection $\{V_{\alpha}\}_{\alpha}$ of affine $P[\sigma]$-varieties and for all pairs $\alpha,\beta$ we have open subsets $V_{\beta\alpha}\subseteq V_{\alpha}$ and isomorphisms $g_{\beta\alpha}\colon V_{\beta\alpha}\simeq V_{\alpha\beta}$ satisfying the following compatibility conditions:
\begin{itemize}
\item $g_{\alpha\beta}=g^{-1}_{\beta\alpha}$ for all $\alpha,\beta$;
\item $g_{\beta\alpha}(V_{\beta\alpha}\cap V_{\gamma\alpha})=V_{\alpha\beta}\cap V_{\gamma\beta}$ and $g_{\gamma\beta}\circ g_{\beta\alpha}=g_{\gamma\alpha}$ on $V_{\beta\alpha}\cap V_{\gamma\alpha}$ for all $\alpha,\beta,\gamma$.
\end{itemize}
Now we can glue together $\{V_{\alpha}\}_{\alpha}$ along open subsets $V_{\alpha\beta}$ through isomorphisms $g_{\beta\alpha}$, and denote it by $X$.
\begin{definition}
The above $X$ is called an {\em abstract $P[\sigma]$-variety}. Its open sets are those subsets that restrict to open subsets in each $V_{\alpha}$. Its closed sets are called {\em subvarieties} of $X$. We say that $X$ is {\em irreducible} if it is not the union of two proper subvarieties.
\end{definition}

\subsection{The Toric $P[\sigma]$-Variety of a Fan}
Now we give the definition of abstract toric $P[\sigma]$-varieties.
\begin{definition}\label{abs-2}
An {\em (abstract) toric $P[\sigma]$-variety} is an irreducible abstract $P[\sigma]$-variety $X$ containing a $\D$-torus $T$ as an open subset such that the action of $T$ on itself extends to a $\D$-algebraic group action of $T$ on $X$.
\end{definition}

It is clear that both affine toric $P[\sigma]$-varieties and projective toric $P[\sigma]$-varieties we have defined in the  previous sections are abstract toric $P[\sigma]$-varieties.

We will construct abstract toric $P[\sigma]$-varieties from fans. First we give the definition of a fan.
\begin{definition}\label{abs-1}
Let $\{S_i\}_i$ be a finite collection of affine $P[x]$-semimodules in $\Z[x]^n$. We say that $\{S_i\}$ is {\em compatible} if it satisfies:
\begin{enumerate}
\item[(a)] $S_i^{md}=M$ for all $i$ and for some $\Z[x]$-lattice $M$;
\item[(b)] for all pairs $(i,j)$ such that $i\neq j$, there exists $\bu\in S_i$ such that $-\bu\in S_j$ and $S_i+\Z[x](-\bu)=S_j+\Z[x](\bu)$;
\item[(c)] for all triples $(i,j,k)$ such that $i\neq j, j\neq k, k\neq i$, by (b), there exist $\bu\in S_i, \bv\in S_j, \bw\in S_k$ such that $S_i+\Z[x](-\bu)=S_j+\Z[x](\bu), S_j+\Z[x](-\bv)=S_k+\Z[x](\bv), S_k+\Z[x](-\bw)=S_i+\Z[x](\bw)$. For such $\bu,\bv,\bw$, $S_i+\Z[x](-\bu)+\Z[x](\bw)=S_j+\Z[x](\bu)+\Z[x](-\bv)=S_k+\Z[x](\bv)+\Z[x](-\bw)$.
\end{enumerate}
\end{definition}
\begin{definition}
A {\em fan} $\Sigma$ is a finite collection of affine $P[\sigma]$-semimodules $\{S_i\}_i$ which is compatible. If $\Sigma$ is a fan, we will denote $\Sigma^{md}=S_i^{md}$ and define $\rank(\Sigma)=\rank(\Sigma^{md})$.
\end{definition}
\begin{example}\label{atpv-exam}
Let $U=\{\bu_1,\ldots,\bu_m\}\subset \Z[x]^n$ and $S_i=\Z[x](U-\bu_i), i=1,\ldots,m$. One can check that $\Sigma=\{S_i\}_{i=1}^m$ satisfies the above compatible conditions and thus is a fan.
\end{example}

We now show how we can construct an abstract toric $P[\sigma]$-variety from a fan. Let $\Sigma$ be a fan. By Theorem \ref{aspv-thm3}, each $S_i$ in $\Sigma$ gives an affine toric $P[\sigma]$-variety $X_i=\Spec^{P[\sigma]}(S_i)$. Let $S_i$ and $S_j$ be two different affine $P[\sigma]$-semimodules in $\Sigma$, then by Definition \ref{abs-1}(b), there exists $\bu\in \Sigma^{md}$, such that $k[S_i]_{\chi^{\bu}}=k[S_j]_{\chi^{-\bu}}$, so we have an isomorphism
\[g_{ji}\colon (X_i)_{\chi^{\bu}}\simeq (X_j)_{\chi^{-\bu}}\]
which is the identity map. For any distinct $i,j,k$, there exist $\bu,\bv,\bw\in \Sigma^{md}$, such that
\begin{align*}
(X_i)_{\chi^{\bu}}\cap (X_i)_{\chi^{-\bw}}&=\Spec^{P[\sigma]}(S_i+\Z[x](-\bu)+\Z[x](\bw)),\\
(X_j)_{\chi^{-\bu}}\cap (X_j)_{\chi^{\bv}}&=\Spec^{P[\sigma]}(S_j+\Z[x](\bu)+\Z[x](-\bv)),\\
(X_k)_{\chi^{-\bv}}\cap (X_k)_{\chi^{\bw}}&=\Spec^{P[\sigma]}(S_k+\Z[x](\bv)+\Z[x](-\bw)).
\end{align*}
Then by Definition \ref{abs-1}(c),
$$(X_i)_{\chi^{\bu}}\cap (X_i)_{\chi^{-\bw}}=(X_j)_{\chi^{-\bu}}\cap (X_j)_{\chi^{\bv}}=(X_k)_{\chi^{-\bv}}\cap (X_k)_{\chi^{\bw}}.$$
So the compatibility conditions for gluing the affine toric $P[\sigma]$-varieties $X_i$ along the open subsets $(X_i)_{\chi^{\bu}}$ are satisfied. Hence we obtain an abstract $P[\sigma]$-variety $X_{\Sigma}$ associated with the fan $\Sigma$.
\begin{theorem}
Let $\Sigma=\{S_i\}_i$ be a fan in $\Z[x]^n$. Then the abstract $P[\sigma]$-variety $X_{\Sigma}$ constructed above is a toric $P[\sigma]$-variety.
\end{theorem}
\begin{proof}
Let $M=\Sigma^{md}$ and $T_{\Sigma}=\Spec^{P[\sigma]}(k[M])$. Then $T_{\Sigma}\subseteq X_i$ as a $\D$-torus for all $i$. These $\D$-tori are all identified by the gluing, so $T_{\Sigma}\subseteq X_{\Sigma} $ as an open subset of $X_{\Sigma}$. For each $i$, $T_{\Sigma}$ has an action on $X_i$. The gluing isomorphisms $g_{ji}$ are identity maps on each $X_i\cap X_j$, so the actions are compatible on each $X_i\cap X_j$, and patch together to give a $\D$-algebraic group action of $T_{\Sigma}$ on $X_{\Sigma}$. $X_{\Sigma}$ is irreducible since all $X_i$ are irreducible affine toric $P[\sigma]$-varieties. So by Definition \ref{abs-2}, $X_{\Sigma}$ is a toric $P[\sigma]$-variety.
\end{proof}
\begin{example}\label{atv-ex1}
Let $U=\{\bu_1,\ldots,\bu_m\}\subset \Z[x]^n$ and $S_i=\Z[x](U-\bu_i), i=1,\ldots,m$. In Example \ref{atpv-exam}, we see that $\Sigma=\{S_i\}_{i=1}^m$ is a fan. So the projective toric $\p$-variety $Y_U$ defined by $U$ is an abstract toric $\p$-variety associated with the fan $\Sigma$.
\end{example}

The irreducible invariant $\p$-subvarieties-faces correspondence (Theorem \ref{oapv-thm1}) still applies to abstract toric $P[\sigma]$-varieties constructed from fans through considering the gluing.

Suppose $\Sigma=\{S_i\}_i$ is a fan in $\Z[x]^n$. Let
\[F_{\Sigma}:=\{F\mid F\preceq S_i, S_i\in \Sigma\}\]
be the set of faces of affine $\p$-semimodules in $\Sigma$ and
\[F_{\Sigma}(r):=\{F\in F_{\Sigma}\mid \rank(F)=\rank(\Sigma)-r\}.\]

Define an equivalence relationship in $F_{\Sigma}$ as follows:
for $F_i\preceq S_i,F_j\preceq S_j$ and $S_i+\Z[x](-\bu)=S_j+\Z[x](\bu)$, $F_i\sim F_j$ if and only if there exists a face $F$ of $S_i+\Z[x](-\bu)$ such that $F_i^{md}=F_j^{md}=F^{md}$.

Let
$L_{\Sigma}:=F_{\Sigma}/\sim$
and
$L_{\Sigma}(r):=F_{\Sigma}(r)/\sim$. For $L',L\in L_{\Sigma}$, $L'\preceq L$ means there is a representative $F'$ of $L'$ and a representative $F$ of $L$ such that $F'\preceq F$.

First, let us prove some lemmas.
\begin{lemma}\label{tvf-lemma}
Let $S$ be an affine $P[x]$-semimodule and F a face of $S$. Then for any $\bu\in F^{md}$, $F+\Z[x](\bu)$ is a face of $S+\Z[x](\bu)$.
\end{lemma}
\begin{proof}
Suppose $\ba,\bb\in S+\Z[x](\bu)$ such that $\ba+\bb\in F+\Z[x](\bu)$. Write $\ba=\ba'+g_1\bu$, $\bb=\bb'+g_2\bu$, and $\ba+\bb=\bc+g_3\bu$, where $\ba',\bb'\in S, \bc\in F, g_1,g_2,g_3\in \Z[x]$. Then $\ba+\bb=\ba'+\bb'+(g_1+g_2)\bu=\bc+g_3\bu$. So $\ba'+\bb'+(g_1+g_2-g_3)\bu=\bc$. Since $(g_1+g_2-g_3)\bu\in F^{md}$, we can write $(g_1+g_2-g_3)\bu=\bd-\be, \bd,\be\in F$. So $\ba'+\bb'+\bd=\bc+\be\in F$. It follows $\ba',\bb'\in F$. Thus $\ba,\bb\in F+\Z[x](\bu)$. Suppose $\ba\in S+\Z[x](\bu)$, $g\in P[x]^*$ such that $g\ba\in F+\Z[x](\bu)$. Write $\ba=\ba'+g_1\bu$, and $g\ba=\bc+g_2\bu$, where $\ba'\in S, \bc\in F, g_1,g_2\in \Z[x]$. Then $g\ba=g\ba'+gg_1\bu=\bc+g_2\bu$. So $g\ba'+(gg_1-g_2)\bu=\bc$. Since $(gg_1-g_2)\bu\in F^{md}$, we can write $(gg_1-g_2)\bu=\bd-\be, \bd,\be\in F$. So $g\ba'+\bd=\bc+\be\in F$. It follows $\ba'\in F$. Thus $\ba\in F+\Z[x](\bu)$. Hence $F+\Z[x](\bu)$ is a face of $S+\Z[x](\bu)$.
\end{proof}
\begin{remark}
From the above lemma, we see that if $\Sigma=\{S_i\}_i$ is a fan and $S_i+\Z[x](-\bu)=S_j+\Z[x](\bu)$, then for $F_i\preceq S_i, F_j\preceq S_j$, $F_i\sim F_j$ if and only if $\bu\in F_i^{md}=F_j^{md}$.
\end{remark}
\begin{lemma}\label{atv-lm1}
Let $F$ be a face of an affine $P[x]$-semimodule $S$. Then $F=S\cap F^{md}$.
\end{lemma}
\begin{proof}
Clearly, $F\subseteq S\cap F^{md}$. We need to show $F\supseteq S\cap F^{md}$. Suppose $S=P[x](\{\bu_1,\ldots,\bu_m\})$ and $F=P[x](\{\bu_1,\ldots,\bu_r\})$. If $\bw\in S\cap F^{md}$, we can write $\bw=\sum_{i=1}^mg_i\bu_i=\sum_{i=1}^rf_i\bu_r$, where $g_i\in P[x]$ and $f_i\in \Z[x]$. Assume $f_i=(f_i)_+-(f_i)_-, (f_i)_+,(f_i)_-\in P[x]$. Then we have $\sum_{i=1}^r(g_i+(f_i)_-)\bu_i+\sum_{i=r+1}^mg_i\bu_i=\sum_{i=1}^r(f_i)_+\bu_r\in F$. Since $F$ is a face of $S$, it follows $\sum_{i=r+1}^mg_i\bu_i\in F$ and if $g_i\neq 0$, then $\bu_i\in F$, $r+1\le i\le m$. Because $\bu_i\notin F, r+1\le i\le m$, then $g_i=0, r+1\le i\le m$. Thus $\bw=\sum_{i=1}^rg_i\bu_i\in F$.
\end{proof}
\begin{lemma}\label{tvf-lemma1}
Let $\Sigma=\{S_i\}_i$ be a fan. Then every $L\in L_{\Sigma}$ is also a fan.
\end{lemma}
\begin{proof}
We need to check that $L$ satisfies the three compatible conditions (a), (b), (c) in Definition \ref{abs-1}. For $F_i, F_j\in L$, assume $F_i\preceq S_i,F_j\preceq S_j$ and $S_i+\Z[x](-\bu)=S_j+\Z[x](\bu)$. Since $F_i\sim F_j$, there exists a face $F$ of $S_i+\Z[x](-\bu)$ such that $F_i^{md}=F_j^{md}=F^{md}$. So (a) is satisfied. Since $F$ is a face of $S_i+\Z[x](-\bu)$, we see $\bu\in F$ and $\bu\in F_i^{md}=F_j^{md}$. To prove (b), because of the symmetry, we just need to show $F_i+\Z[x](-\bu)\subseteq F_j+\Z[x](\bu)$, or $F_i\subseteq F_j+\Z[x](\bu)$. Since $F_i\subseteq S_i\subseteq S_j+\Z[x](\bu)$, for $\ba\in F_i$, we can write $\ba=\bb+g\bu$, where $\bb\in S_j, g\in \Z[x]$. Then $\bb=\ba-g\bu\in F_j^{md}\cap S_j$. Therefore, by Lemma \ref{atv-lm1}, $\bb\in F_j$. Hence $\ba\in F_j+\Z[x](\bu)$. So (b) is satisfied. To prove (c), suppose $F_i\sim F_j\sim F_k\in L$ and $F_i\preceq S_i,F_j\preceq S_j,F_k\preceq S_k$. Assume $\bu,\bv,\bw\in \Sigma^{md}$ such that $S_i+\Z[x](-\bu)+\Z[x](-\bv)=S_j+\Z[x](\bu)+\Z[x](-\bw)=S_k+\Z[x](\bv)+\Z[x](\bw)$. As above, $\bu,\bv,\bw\in F_i^{md}=F_j^{md}=F_k^{md}$. For the symmetry, we just need to prove $F_i+\Z[x](-\bu)+\Z[x](-\bv)\subseteq F_j+\Z[x](\bu)+\Z[x](-\bw)$, or $\Z[x](-\bv)\subseteq F_j+\Z[x](\bu)+\Z[x](-\bw)$. Suppose $g\bv\in \Z[x](-\bv)$. Since $g\bv\in S_j+\Z[x](\bu)+\Z[x](-\bw)$, we can write $g\bv=\bb+h_1\bu+h_2\bw$, where $\bb\in S_j, h_1,h_2\in\Z[x]$. Then $\bb=g\bv-h_1\bu-h_2\bw\in F_j^{md}\cap S_j$. Therefore, by Lemma \ref{atv-lm1} again, $\bb\in F_j$. Hence $g\bv\in F_j+\Z[x](\bu)+\Z[x](-\bw)$. So (c) is satisfied. Hence $L$ is a fan.
\end{proof}

Now we give the irreducible invariant $\p$-subvarieties-faces correspondence theorem.
\begin{theorem}\label{tvf-2}
Let $X_{\Sigma}$ be the toric $P[\sigma]$-variety associated with fan $\Sigma=\{S_i\}_i$ and assume $T_{\Sigma}$ is the $\D$-torus of $X_{\Sigma}$. Then there is a one-to-one correspondence between elements of $L_{\Sigma}$ and irreducible $T_{\Sigma}$-invariant $\p$-subvarieties of $X_{\Sigma}$. Let $L\in L_{\Sigma}$. If we denote the irreducible $T_{\Sigma}$-invariant $\p$-subvariety associated with $L$ by $D(L)$, then $D(L)\simeq X_L$ which is the toric $P[\sigma]$-variety associated with the fan $L$.
\end{theorem}
\begin{proof}
For an element $L=\{F_i\}_i$ of $L_{\Sigma}$, each $F_i\preceq S_i$ corresponds to an irreducible $T_{\Sigma}$-invariant $\p$-subvariety of $X_i=\Spec^{\p}(k[S_i])$. The gluing of $X_i$ along open subsets $(X_i)_{\chi^{\bu}}$ induces a gluing of $X_{F_i}=\Spec^{\p}(k[F_i])$ along open subsets $(X_{F_i})_{\chi^{\bu}}$. The resulted $\p$-variety is exactly $X_L$ since $L$ is a fan by Lemma \ref{tvf-lemma1}.

For the converse, suppose $Y$ is an irreducible $T_{\Sigma}$-invariant $\p$-subvariety of $X_{\Sigma}$. Then $Y\cap X_i$ is an irreducible $T_{\Sigma}$-invariant $\p$-subvariety of $X_i$, thus there exists a face $F_i$ of $S_i$ such that $Y\cap X_i=\Spec^{\p}(k[F_i])$. The gluing of $Y\cap X_i$ is induced by the gluing of $X_i$. Therefore if $S_i+\Z[x](-\bu)=S_j+\Z[x](\bu)$, then $F_i+\Z[x](-\bu)=F_j+\Z[x](\bu)$, and $\bu\in F_i^{md}=F_j^{md}$. Let $F=F_i+\Z[x](-\bu)$ which is a face of $S_i+\Z[x](-\bu)$ by Lemma \ref{tvf-lemma}. Obviously, $F^{md}=F_i^{md}=F_j^{md}$. Therefore, $F_i\sim F_j$. So $L=\{F_i\}_i$ is an element of $L_{\Sigma}$.
\end{proof}

%\begin{theorem}\label{tvf-1}
%Let $\Sigma=\{S_i\}$ be a fan. Assume that every $S_i\in\Sigma$ is face-saturated. Let $X_{\Sigma}$ be the toric $P[\sigma]$-variety of $\Sigma$ and $T_{\Sigma}$ the $P[\sigma]$-torus of $X_{\Sigma}$. Then there is a one-to-one correspondence between elements of $L_{\Sigma}$ and $T_{\Sigma}$-orbits of $X_{\Sigma}$. Let $L\in L_{\Sigma}$. If we denote the orbit associated with $L$ by $O(L)$, then $O(L)\simeq \Spec^{P[\sigma]}(L^{md})$ which is a $P[\sigma]$-torus of $\sigma$-dimension $\rank(L)$. Moreover,
%\[\overline{O(L)}=D(L)=\bigcup_{L'\in L_{\Sigma},L'\preceq L}O(L').\]
%\end{theorem}
%\begin{proof}
%By Theorem(\ref{oapv-thm}), each orbit of any $X_{S_i}$ exactly corresponds to a face of $S_i$. If $F_i\sim F_j\in F_{\Sigma}$, assume $F_i\preceq S_i, F_j\preceq S_j$ and $S_i+\Z[x](-u)=S_j+\Z[x](u)$. Since there is a face $F$ of $S_i+\Z[x](-u)$ such that $F^{md}=F_i^{md}=F_j^{md}$, the orbits corresponding to $F_i$ and $F_j$ glue to be the same orbit. Thus there is a one-to-one correspondence between elements of $L_{\Sigma}$ and $T_{\Sigma}$-orbits of $X_{\Sigma}$. Use $\sim$ to mean gluing, then
%\[\overline{O(L)}=\bigcup_{F_i\in L}\bigcup_{F_{ij}\preceq F_i}O(F_{ij})/\!\!\sim\,\,=\bigcup_{F_i\in L}D(F_i)/\!\!\sim\,\,=D(L)=\bigcup_{L'\in L_{\Sigma},L'\preceq L}O(L').\]
%\end{proof}

\section{Divisors on Toric $P[\sigma]$-Varieties}
In algebraic geometry, the divisor theory is a very useful tool to study the properties of algebraic varieties. In this section, we will define divisors and divisor class modules for toric $P[\sigma]$-varieties by virtue of the irreducible invariant $\p$-subvarieties-faces correspondence. Moreover, we will establish connections between the properties of toric $P[\sigma]$-varieties and divisor class modules.

First let us give a description of faces of affine $P[x]$-semimodules using supporting hyperplanes. Let $S$ be an affine $P[x]$-semimodule and $M=S^{md}$. Let $N=\Hom_{\Z[x]}(M,\Z[x])$. Given $\varphi\in N, \varphi\neq0$, let
\[H_{\varphi}:=\{\bu\in M\mid \varphi(\bu)=0\}\subseteq M\]
which is called the {\em hyperplane} defined by $\varphi$ and
\[H_{\varphi}^+:=\{\bu\in M\mid \varphi(\bu)\geqslant0\}\subseteq M\]
which is called the {\em closed half-space} defined by $\varphi$. If $S\subseteq H_{\varphi}^+$, then $H_{\varphi}$ is called a {\em supporting hyperplane} of $S$ and $H_{\varphi}^+$ is called a {\em supporting half-space}.
\begin{prop}
Let $S$ be an affine $P[x]$-semimodule and $M=S^{md}$. Then $F$ is a proper face of $S$ if and only if there exists a $\varphi\in N=\Hom_{\Z[x]}(M,\Z[x])$ such that $H_{\varphi}^+$ is a supporting half-space of $S$ and $F=H_{\varphi}\cap S$.
\end{prop}
\begin{proof}
``$\Rightarrow$''. Suppose $S=P[x](U)=P[x](\{\bu_1,\ldots,\bu_m\})\in\Z[x]^n$ and $\rank(S)=t$. Let $L=\Syz(U)$ and $\rank(L)=m-t$ by Lemma \ref{pzl-lemma2}. Then the map $\be_i\mapsto\bu_i$ gives the following exact sequence:
\[0\longrightarrow L\longrightarrow \Z[x]^m\longrightarrow M\longrightarrow0.\]
Without loss of generality, assume that $F=P[x](\{\bu_1,\ldots,\bu_r\})$ and $\rank(F)=s<t$. Let $V=\{\varphi\in N\mid F=H_{\varphi}\cap S\}$. If $A=\{\ba_1,\ldots,\ba_{m-t}\}$ is a basis of $L$ and regard $A$ as a matrix with columns $\ba_i$, then
$$N\simeq\{\varphi=(\varphi_1,\ldots,\varphi_m)^{\tau}\in \Z[x]^m\mid A\varphi=\mathbf{0}\}$$
and
$$V\simeq\{\varphi=(0,\ldots,0,\varphi_{r+1},\ldots,\varphi_m)^{\tau}\in \Z[x]^m\mid A\varphi=\mathbf{0}\}.$$
So $V$ is a free $\Z[x]$-module and $\rank(V)=(m-r)-((m-t)-(r-s))=t-s>0$ by Lemma \ref{pzl-lemma2}. We can assume that $A$ is a trapezoidal matrix. Because $F$ is a face, we can choose $\varphi\in V$ such that $\varphi(\bu_i)\neq 0, i=r+1,\ldots,m$. Then $F=H_{\varphi}\cap S$. We claim that $\varphi(\bu_i)$ have the same sign for $i=r+1,\ldots,m$. Otherwise, suppose that $\varphi(\bu_i)>0$ and $\varphi(\bu_j)<0$, $i,j\geqslant r+1$. Then $\varphi(\varphi(\bu_i)\bu_j-\varphi(\bu_j)\bu_i)=0$, thus $\varphi(\bu_i)\bu_j-\varphi(\bu_j)\bu_i\in F$ and it follows $\bu_i,\bu_j\in F$ which is a contradiction. Therefore, without loss of generality, we can assume that $\varphi(\bu_i)>0, i=r+1,\ldots,m$ and hence $H_{\varphi}^+$ is a supporting half-space of $S$.

``$\Leftarrow$''. Suppose $H_{\varphi}^+$ is a supporting half-space of $S$ and $F=H_{\varphi}\cap S$, where $\varphi\in N, \varphi\neq 0$. It is clear that $F$ is a $P[x]$-semimodule. Let $\bu_1,\bu_2\in S, g_1,g_2\in P[x]^*$ such that $g_1\bu_1+g_2\bu_2\in F$. Then $\varphi(g_1\bu_1+g_2\bu_2)=g_1\varphi(\bu_1)+g_2\varphi(\bu_2)=0$. Since $\bu_1,\bu_2\in S$, $\varphi(\bu_1)\ge0$ and $\varphi(\bu_2)\ge 0$. We must have $\varphi(\bu_1)=\varphi(\bu_2)=0$. It follows $\bu_1,\bu_2\in F$ and hence $F$ is a face as desired.
\end{proof}

Suppose $S$ is an affine $P[x]$-semimodule. If $F$ is a proper face of $S$, denote $V(F):=\{\varphi\in N\mid F=H_{\varphi}\cap S\}$. From the proof of the above proposition, we know that $V(F)$ is a free $\Z[x]$-module and $\rank(V(F))=\rank(S)-\rank(F)$. In particular, if $F$ is a facet of $S$, then $V(F)$ is a free $\Z[x]$-module of rank one which has a basis $\{\varphi\}$ with $S\subseteq H_{\varphi}^+$. In this case, we call $\varphi$ the {\em standard normal vector} of $F$.
%If $\{F_i\}_{i=1}^s$ are all facets of $S$ and $\{\varphi_i\}_{i=1}^s$ are their standard normal vectors respectively, then:
%\begin{prop}
%\[S=H_{\varphi_1}^+\cap\ldots\cap H_{\varphi_s}^+.\]
%\end{prop}
%\begin{prop}
%Every proper face $F$ of $S$ is the intersection of all facets of $S$ containing $F$.
%\end{prop}

Let $\Sigma=\{S_i\}_i$ be a fan and $M=\Sigma^{md}$. We will define a valuation on $k(M)$ for every facet of $S_i$. First consider the affine case. Let $S$ be an affine $P[x]$-semimodule and $F$ a facet of $S$. Assume $\varphi_F$ is the standard normal vector of $F$. Define a {\em valuation} $\nu_F$ on $k[S]$ as follows:
\[\nu_F\colon k[S]\to \Z[x], f=\sum_{\bu}\alpha_{\bu}\chi^{\bu}\mapsto \min\limits_{\bu}(\varphi_F(\bu)), f\in k[S].\]
Extend the valuation to $k(S):=\Frac(k[S])$ by defining $\nu_F(\frac{f}{g})=\nu_F(f)-\nu_F(g)$, for $\frac{f}{g}\in k(S)$. Note that for $\bu\in S, \nu_F(\chi^{\bu})=\varphi_F(\bu)$ and if $\bu\in F, \nu_F(\chi^{\bu})=0$, if $\bu\in S\backslash F, \nu_F(\chi^{\bu})>0$.
%Easy to check $k[S]_{k[S\backslash F]}$ is the $\sigma$-valuation ring associated with this valuation.

Now let $\Sigma=\{S_i\}_i$ be a fan in $\Z[x]^n$. For any $L\in L_{\Sigma}(1)$, if $L$ has only one element $F$, then $F$ is a facet of some $S_i$, and we define $\nu_L:=\nu_F$; if $L$ has more than one element, choose one for example $F_i$ and we define $\nu_L:=\nu_{F_i}$. In the latter case, for $F_i\sim F_j\in L$, suppose $F_i\preceq S_i, F_j\preceq S_j$ and $S_i+\Z[x](-\bu)=S_j+\Z[x](\bu)$. By the definition of the equivalence relationship $\sim$, there exists a facet $F$ of $S_i+\Z[x](-\bu)$ such that $F_i^{md}=F_j^{md}=F^{md}$. It follows that $F_i$, $F_j$ and $F$ have the same standard normal vector and hence $\nu_{F_i}=\nu_{F_j}=\nu_F$. So $\nu_L$ is independent of the choice of $F_i$.

Let $X_{\Sigma}$ be the toric $P[\sigma]$-variety associated with the fan $\Sigma$ and assume $T_{\Sigma}$ is the $\D$-torus of $X_{\Sigma}$. By Theorem \ref{tvf-2}, each $L\in L_{\Sigma}(1)$ corresponds to a $\sigma$-codimension one irreducible $T_{\Sigma}$-invariant $\p$-subvariety $D(L)$ of $X_{\Sigma}$, which is called a {\em prime divisor} of $X_{\Sigma}$.
\begin{definition}
$\Div(X_{\Sigma})$ is the free $\Z[x]$-module generated by the prime divisors of $X_{\Sigma}$ as a basis. A {\em Weil divisor} is an element of $\Div(X_{\Sigma})$ which is of the form $\sum_{L\in L_{\Sigma}(1)}a_LD_L$.
\end{definition}

Let $D=\sum_{L\in L_{\Sigma}(1)}a_LD_L$, then $D$ is said to be {\em effective}, written as $D\ge 0$, if $a_L\ge 0$ for all $L$.
\begin{definition}
Let $\Sigma=\{S_i\}_i$ be a fan and $X_{\Sigma}$ the toric $P[\sigma]$-variety associated with $\Sigma$. Assume $M=\Sigma^{md}$.
\begin{enumerate}
\item The {\em divisor} of $f\in k(M):=\Frac(k[M])$ is defined to be $\divsor(f)=\sum_L\nu_L(f)D_L$ where $L\in L_{\Sigma}(1)$.
\item A divisor of the form $\divsor(f)$ for some $f\in k(M)$ is called a {\em principal divisor}, and the set of all principal divisors is denoted by $\Div_0(X_{\Sigma})$.
\item For $\bu\in M$, $\divsor(\chi^{\bu})=\sum_{L\in L_{\Sigma}(1)}\varphi_L(\bu)D_L$ is called a {\em characteristic divisor}, and the set of all characteristic divisors is denoted by $\Div_c(X_{\Sigma})$.
\item Divisors $D$ and $E$ are said to be {\em linearly equivalent}, written $D\sim E$, if their difference is a principal divisor, i.e.\ $D-E=\divsor(f)\in \Div_0(X_{\Sigma})$ for some $f\in k(M)$.
\end{enumerate}
\end{definition}

If $f,g\in k(M)$, then $\divsor(fg)=\divsor(f)+\divsor(g)$ and $\divsor(f^a)=a\,\divsor(f)$, for $a\in \Z[x]$. Thus $\Div_0(X_{\Sigma})$ is a $\Z[x]$-submodule of $\Div(X_{\Sigma})$. Similarly, $\Div_c(X_{\Sigma})$ is also a $\Z[x]$-submodule of $\Div(X_{\Sigma})$.

If $D=\sum_ia_iD_i$ is a Weil divisor on $X_{\Sigma}$ and $U\subseteq X_{\Sigma}$ is a nonempty open subset, then $D|_U=\sum_{D_i\cap U\neq \emptyset}a_iD_i\cap U$ is called the {\em restriction} of $D$ on $U$.
\begin{definition}
A Weil divisor $D$ on a toric $P[\sigma]$-variety $X_{\Sigma}$ is {\em Cartier} if it is {\em locally characteristic}, meaning that $X_{\Sigma}$ has an open cover $\{U_i\}_{i\in I}$ such that $D|_{U_i}$ is characteristic on $U_i$ for every $i\in I$, namely there exists $\bu_i\in M$, such that $D|_{U_i}=\divsor(\chi^{\bu_i})|_{U_i}$ for $i\in I$, and we call $\{(U_i,\bu_i)\}_{i\in I}$ the {\em local data} for $D$.
\end{definition}

We can check that all of Cartier divisors on $X_{\Sigma}$ form a $\Z[x]$-module $\CDiv(X_{\Sigma})$ satisfying $\Div_c(X_{\Sigma})\subseteq \CDiv(X_{\Sigma})\subseteq \Div(X_{\Sigma})$.
\begin{definition}
Let $X_{\Sigma}$ be the toric $P[\sigma]$-variety of a fan $\Sigma=\{S_i\}_i$. Its {\em class module} is
\[\Cl(X_{\Sigma}):=\Div(X_{\Sigma})/\Div_c(X_{\Sigma})\]
and its {\em Picard module} is
\[\Pic(X_{\Sigma}):=\CDiv(X_{\Sigma})/\Div_c(X_{\Sigma}).\]
\end{definition}

Obviously, $\Pic(X_{\Sigma})\hookrightarrow\Cl(X_{\Sigma})$ as $\Z[x]$-modules.

\begin{example}
Let $U=\{\be_1,\ldots,\be_n\}$ be the standard basis of $\Z[x]^n$. Obviously, the toric $\p$-variety $X_U$ defined by $U$ is just the $\sigma$-affine space $\A^n$. The affine $P[x]$-semimodule $S=P[x](U)=P[x]^n$ has $n$ facets whose standard normal vectors are $\{\varphi_1,\ldots,\varphi_n\}$ with $\varphi_i(\be_i)=1, \varphi_i(\be_j)=0, j\neq i$, $1\le i,j\le n$. Suppose the corresponding prime divisors are $\{D_1,\ldots,D_n\}$ respectively. Then for each $i$, $\divsor(\chi^{\be_i})=D_i$. Thus $$\Cl(\A^n)=\Z[x]D_1\oplus\cdots\oplus\Z[x]D_n/(D_1,\ldots,D_n)=0.$$
\end{example}

\begin{example}
Suppose $U=\{\bu_1=(x,1),\bu_2=(x,2),\bu_3=(x,3)\}$ and $X_U$ is the affine toric $\p$-variety defined by $U$. $S=P[x](U)$ has two facets $F_1=P[x](\{\bu_1\})$ and $F_2=P[x](\{\bu_3\})$, and we denote their corresponding prime divisors by $D_1$ and $D_2$ respectively. The standard normal vector of $F_1$ is $\varphi_1$ with $\varphi_1(\bu_1)=0,\varphi_1(\bu_2)=1,\varphi_1(\bu_3)=2$. The standard normal vector of $F_2$ is $\varphi_2$ with $\varphi_2(\bu_1)=2,\varphi_1(\bu_2)=1,\varphi_1(\bu_3)=0$. So $\divsor(\chi^{\bu_1})=2D_2$, $\divsor(\chi^{\bu_2})=D_1+D_2$ and $\divsor(\chi^{\bu_3})=2D_1$. Thus
\[\Cl(X_U)=\Z[x]D_1\oplus\Z[x]D_2\oplus\Z[x]D_3/(2D_2,D_1+D_2,2D_1)\simeq \Z[x]/(2).\]
\end{example}
\begin{example}
Suppose $U=\{\bu_1=(x,1,1),\bu_2=(1,x,1),\bu_3=(1,1,x),\bu_4=(1,1,1)\}$ and $X_U$ is the affine toric $\p$-variety defined by $U$. $S=P[x](U)$ has three facets $F_1=P[x](\{\bu_2,\bu_3\})$ and $F_2=P[x](\{\bu_1,\bu_3\})$, $F_3=P[x](\{\bu_1,\bu_2\})$, and we denote their corresponding prime divisors by $D_1$, $D_2$ and $D_3$ respectively. The standard normal vector of $F_1$ is $\varphi_1$ with $\varphi_1(\bu_1)=x+2,\varphi_1(\bu_2)=\varphi(\bu_3)=0,\varphi_1(\bu_4)=1$. The standard normal vector of $F_2$ is $\varphi_2$ with $\varphi_2(\bu_1)=\varphi(\bu_3)=0,\varphi_1(\bu_2)=x+2,\varphi_1(\bu_4)=1$. The standard normal vector of $F_3$ is $\varphi_3$ with $\varphi_2(\bu_1)=\varphi(\bu_2)=0,\varphi_1(\bu_3)=x+2,\varphi_1(\bu_4)=1$. So $\divsor(\chi^{\bu_1})=(x+2)D_1$, $\divsor(\chi^{\bu_2})=(x+2)D_2$, $\divsor(\chi^{\bu_3})=(x+2)D_3$ and $\divsor(\chi^{\bu_4})=D_1+D_2+D_3$. Thus
\begin{align*}
\Cl(X_U)&=\bigoplus_{i=1}^4\Z[x]D_i/((x+2)D_1,(x+2)D_2,(x+2)D_3,D_1+D_2+D_3)\\
&\simeq \Z[x]/(x+2)\oplus\Z[x]/(x+2).
\end{align*}
\end{example}

\begin{example}
The $\D$-projective space $\PP^2$ is defined by $U=\{\mathbf{0},\be_1,\be_2\}$, where $\{\be_1,\be_2\}$ is the standard basis of $\Z[x]^2$. The affine $\p$-semimodules associated with $\PP^2$ are $S_1=P[x](\{\be_1,\be_2\})$, $S_2=P[x](\{-\be_1,\be_2-\be_1\})$ and $S_3=P[x](\{-\be_2,\be_1-\be_2\})$. $S_1$ has facets $F_1=P[x](\be_1)$ and $F_2=P[x](\be_2)$. $S_2$ has facets $F_3=P[x](-\be_1)$ and $F_4=P[x](\be_2-\be_1)$. $S_3$ has facets $F_5=P[x](-\be_2)$ and $F_6=P[x](\be_1-\be_2)$. $F_1\sim F_3$ have the standard normal vector $\varphi_1$ with $\varphi_1(\be_1)=0,\varphi_1(\be_2)=1$. $F_2\sim F_5$ have the standard normal vector $\varphi_2$ with $\varphi_2(\be_1)=1,\varphi_2(\be_2)=0$. $F_4\sim F_6$ have the standard normal vector $\varphi_3$ with $\varphi_3(\be_1)=-1,\varphi_3(\be_2)=-1$. Denote the corresponding prime divisors by $D_1,D_2,D_3$ respectively. Then $\divsor(\chi^{\be_1})=D_2-D_3$, $\divsor(\chi^{\be_2})=D_1-D_3$. Thus $$\Cl(\PP^2)=\Z[x]D_1\oplus\Z[x]D_2\oplus\Z[x]D_3/(D_2-D_3,D_1-D_3)\simeq\Z[x].$$
In the same way, we can show that $\Cl(\PP^n)\simeq \Z[x], n\geq 1$.
\end{example}

\begin{prop}\label{dtpv-prop}
Let $X=\Spec^{P[\sigma]}(k[S])$ be an affine toric $P[\sigma]$-variety. Then:
\begin{enumerate}
\item[(a)] Every Cartier divisor on $X$ is a characteristic divisor;
\item[(b)] $\Pic(X)=0$.
\end{enumerate}
\end{prop}
\begin{proof}
(a) Let $H$ be the intersection of all faces of $S$ which is still a face of $S$. Then $D(H)\subseteq \bigcap_{F\in F_S(1)}D_F$.
Fix a point $p\in D(H)$. Since $D$ is Cartier, it is locally characteristic, and in particular is characteristic in a neighbourhood $U$ of $p$, i.e.\ $D|_U=\divsor(\chi^{\bu})|_U$ for some $\bu\in S^{md}$. Since $p\in U\cap D_F$ for all $F\in F_S(1)$, $D=\divsor(\chi^{\bu})$.

(b) It follows from (a).
\end{proof}

An affine $P[x]$-semimodule $S$ is said to be {\em compact} if $\bigcap_{F\in F_S(1)}F^{md}=\{\mathbf{0}\}$.
\begin{prop}
Let $X_{\Sigma}$ be the toric $P[\sigma]$-variety of a fan $\Sigma$ and $M=\Sigma^{md}$. If $\Sigma$ contains a compact affine $P[x]$-semimodule, then $\Pic(X_{\Sigma})$ is a $\Z[x]$-lattice.
\end{prop}
\begin{proof}
Because of Lemma \ref{pzl-lemma}, it suffices to show that if $D$ is a Cartier divisor and $gD$ is the divisor of a character for some $g\in P[x]^*$, then the same is true for $D$. Write $D=\sum_La_LD_L$ and assume that $gD=\divsor(\chi^{\bu}), \bu\in M$. Let $S$ is a compact affine $P[x]$-semimodule in $\Sigma$. Since $D$ is Cartier, its restriction to $X_S:=\Spec^{\p}(k[S])$ is also Cartier. Assume that $D|_{X_S}=\sum_{F\in F_S(1)}a_FD_F$. This is characteristic on $X_S$ by the Proposition \ref{dtpv-prop}, so there is a $\bu'\in M$ such that $D|_{X_S}=\divsor(\chi^{\bu'})|_{X_S}$. This implies that $a_F=\varphi_F(\bu')$, for all $F\in F_S(1)$. On the other hand, $gD=\divsor(\chi^{\bu})$ implies that $ga_L=\varphi_L(\bu)$, for all $L\in L_{\Sigma}(1)$. It follows that $\varphi_F(g\bu')=ga_F=\varphi_F(\bu)$, for all $F\in F_S(1)$, i.e.\ $\varphi_F(g\bu'-\bu)=0$ for all $F\in F_S(1)$. So there exists $g_F\in P[x]^*$ such that $g_F(g\bu'-\bu)\in F^{md}$, for all $F\in F_S(1)$. Thus $(\prod_{F\in F_S(1)}g_F)(g\bu'-\bu)\in \bigcap_{F\in F_S(1)}F^{md}$. Since $S$ is compact, $\bigcap_{F\in F_S(1)}F^{md}=\{\mathbf{0}\}$ and hence $g\bu'-\bu=\mathbf{0}$. It follows that $D=\divsor(\chi^{\bu'})$.
\end{proof}

\begin{definition}
A toric $\p$-variety $X$ is said to be {\em smooth} if $\Cl(X)=\Pic(X)$.
\end{definition}
\begin{definition}
Suppose $S$ is an affine $P[x]$-semimodule and $M=S^{md}$. For a facet $F$ of $S$, assume $\varphi_F$ is the standard normal vector of $F$. If $\{\varphi_F\mid F\in F_S(1)\}$ forms a basis of the free module $N=\Hom_{\Z[x]}(M,\Z[x])$, we say $S$ is {\em smooth}. Let $\Sigma$ be a fan. If for every $S\in\Sigma$, $S$ is smooth, then we say $\Sigma$ is {\em smooth}.
\end{definition}

In the algebraic case, the smoothness of a toric variety is equivalent to the smoothness of the corresponding fan. We will generalize this result to the difference case. Firstly, let us prove some lemmas.
\begin{lemma}\label{dtpv-lemma}
Let $X_{\Sigma}$ be the toric $P[\sigma]$-variety of a fan $\Sigma$ and $Z=Spec^{\p}(k[S])$ for some $S\in \Sigma$. Let $D_1,\ldots,D_s$ be the irreducible components of $X\backslash Z$ that are prime divisors. Then the sequence
\[\sum_{j=1}^s\Z[x]D_j\longrightarrow\Cl(X)\longrightarrow\Cl(Z)\longrightarrow 0\]
is exact, where the first map sends $\sum_{j=1}^sa_jD_j$ to its divisor class in $\Cl(X)$ and the second is induced by restriction to $Z$.
\end{lemma}
\begin{proof}
Let $D'=\sum_ia_iD_i'\in \Cl(Z)$ with $D_i'$ a prime divisor in $Z$. Then the closure $\overline{D_i'}$ of $D_i'$ in $X$ is a prime divisor of $X$ and $D=\sum_ia_i\overline{D_i'}$ satisfies $D|_Z=D'$. Hence $\Cl(X)\rightarrow\Cl(Z)$ is surjective.

Since each $D_j$ restricts to $0$ in $\Div(Z)$, the composition of two maps is trivial. To prove the exactness, suppose that $[D]\in \Cl(X)$ restricts to $0$ in $\Cl(Z)$. This means that $D|_Z$ is the divisor of some $\chi^{\bu}\in k[S^{md}]=k[\Sigma^{md}]$, i.e.\ $D|_Z=\divsor(\chi^{\bu})|_Z$. This implies that $D-\divsor(\chi^{\bu})$ is supported on $X\backslash Z$, i.e.\ $D-\divsor(\chi^{\bu})\in \sum_{j=1}^s\Z[x]D_j$, So $[D]\in [\sum_{j=1}^s\Z[x]D_j]$ as desired.
\end{proof}

\begin{lemma}\label{dtpv-lm}
Let $M$ be a free $\Z[x]$-lattice and $N=\Hom_{\Z[x]}(M,\Z[x])$ its dual $\Z[x]$-lattice. For a subset $\{\varphi_1,\ldots,\varphi_s\}\subseteq N$, define a map
\[\Phi\colon \Z[x]^s\longrightarrow N,(a_1,\ldots,a_s)\mapsto\sum_{i=1}^sa_i\varphi_i.\]
The dual map of $\Phi$ is
\[\Phi^*\colon\Hom_{\Z[x]}(N,\Z[x])\simeq M\longrightarrow\Hom_{\Z[x]}(\Z[x]^s,\Z[x])\simeq \Z[x]^s\]
with $\Phi^*(\bm)\colon \Z[x]^s\rightarrow \Z[x],\ba\mapsto \Phi(\ba)(\bm)$, for every $\bm\in M$. Then $\Phi$ is an isomorphism if and only if $\Phi^*$ is an isomorphism.
\end{lemma}
\begin{proof}
Suppose $\Phi^*$ is an isomorphism. For any $\ba=(a_1,\ldots,a_s)\in \ker(\Phi)$, and $\bm\in M$, $\Phi^*(\bm)(\ba)=\Phi(\ba)(\bm)=0$, i.e., $\Phi^*(M)(\ba)=0$. Since $\Phi^*$ is surjective, we have $\Phi^*(M)=\Hom_{\Z[x]}(\Z[x]^s,\Z[x])$. So $\ba=\mathbf{0}$. Thus $\ker(\Phi)=\mathbf{0}$ and $\Phi$ is injective. Since $\Phi^*$ is an isomorphism, $\rank(N)=\rank(M)=s$. Therefore, $\{\varphi_1,\ldots,\varphi_s\}$ is linearly independent and generates $\Span_{\Q(x)}(N)$ as a basis. To prove $\Phi$ is also surjective, we only need to show $\Phi(\Z[x]^s)$ is $\Z[x]$-saturated, i.e., for any $g\in\Z[x]^*$ and $\varphi\in N$, if $g\varphi\in \Phi(\Z[x]^s)$, then $\varphi\in \Phi(\Z[x]^s)$. Assume $\ba=(a_1,\ldots,a_s)$ and $\Phi(\ba)=g\varphi$. Let $\{\be_i\}_{i=1}^s$ be the standard basis of $\Z[x]^s$. Since $\Phi^*$ is surjective, for every $i$, there exists an $\bm_i\in M$ such that $\Phi^*(\bm_i)(\be_i)=1, \Phi^*(\bm_i)(\be_j)=0, j\neq i$. Then $g\varphi(\bm_i)=\Phi(\ba)(\bm_i)=\sum_{i=1}^s\ba_i\varphi(\bm_i)=\ba_i$. Denote $\ba'=(\varphi(\bm_1),\ldots,\varphi(\bm_s))$. Then $g\Phi(\ba')=\Phi(\ba)=g\varphi$ and $\Phi(\ba')=\varphi$. Hence $\Phi$ is an isomorphism.

Since $\Phi^{**}=\Phi$, the converse follows easily.
\end{proof}

\begin{theorem}\label{dtpv-thm}
Let $X_{\Sigma}$ be the toric $P[\sigma]$-variety of a fan $\Sigma$. Assume $M=\Sigma^{md}$ is a free $\Z[x]$-module. Then $X_{\Sigma}$ is smooth if and only if $\Sigma$ is smooth.
\end{theorem}
\begin{proof}
By definition, $X_{\Sigma}$ is smooth if and only if $\Cl(X)=\Pic(X)$ which is equivalent to the fact that every Weil divisor on $X_{\Sigma}$ is Cartier. For the necessity, assume every Weil divisor on $X_{\Sigma}$ is Cartier. By Lemma \ref{dtpv-lemma}, for any $S\in \Sigma$, every Weil divisor on $X_S:=\Spec^{\p}(k[S])$ is Cartier. Then by Proposition \ref{dtpv-prop}, $\Cl(X_S)=\Pic(X_S)=0$. Since we have the following exact sequence:
\[0\longrightarrow M\stackrel{\theta}\longrightarrow \Div(X_S)\longrightarrow \Cl(X_S)\longrightarrow 0,\]
where $\theta$ maps $\bu\in M$ to the divisor of $\chi^{\bu}$, we know that $\theta$ is an isomorphism. If $\{\varphi_F\mid F\in F_S(1)\}=\{\varphi_1,\ldots,\varphi_s\}$, this map becomes
\begin{equation}\label{dtpv-equ}
\theta\colon M\longrightarrow \Z[x]^s, \bm\mapsto (\varphi_1(\bm),\ldots,\varphi_s(\bm)), \forall \bm\in M.
\end{equation}
Now define $\Phi\colon \Z[x]^s\rightarrow N$ by $\Phi(a_1,\ldots,a_s)=\sum_{i=1}^sa_i\varphi_i$. The dual map
\[\Phi^*\colon M=\Hom_{\Z[x]}(N,Z[x])\longrightarrow\Hom_{\Z[x]}(\Z[x]^s,\Z[x])\simeq \Z[x]^s\]
is easily seen to be (\ref{dtpv-equ}). Since $\Phi^*$ is an isomorphism, $\Phi$ is an isomorphism by Lemma \ref{dtpv-lm}. The injectivity of $\Phi$ implies that $\{\varphi_1,\ldots,\varphi_s\}$ is linearly independent. The surjectivity of $\Phi$ implies that $\{\varphi_1,\ldots,\varphi_s\}$ generates $N$ as a $\Z[x]$-module. So $\{\varphi_1,\ldots,\varphi_s\}$ is a basis of $N$ and $S$ is smooth. Thus $\Sigma$ is smooth.

Every step in the above proof is invertible, so the sufficiency follows.
\end{proof}
\begin{example}
The $\D$-projective space $\PP^n$ is defined by $U=\{\mathbf{0},\be_1,\ldots,\be_n\}$, where $\{\be_i\}_{i=1}^n$ is the standard basis of $\Z[x]^n$. Let $S_0=P[x](\{\be_1,\ldots,\be_n\})$ and $S_i=P[x](U-\be_i),i=1,\ldots,n$. The fan associated with $\PP^n$ is $\{S_i\}_{i=0}^n$ by Example \ref{atv-ex1}. It is easy to check that $S_0=P[x](\{\be_1,\ldots,\be_n\})$ and $S_i=P[x](U-\be_i),i=1,\ldots,n$ are smooth. So $\PP^n$ is smooth by Theorem \ref{dtpv-thm} and $\Pic(\PP^n)=\Cl(\PP^n)\simeq \Z[x]$.
\end{example}

\section{Conclusions}
In this paper, we first introduce the concept of $\p$-varieties and initiate the study of toric $\p$-varieties. We define affine toric $\p$-varieties and establish connections between affine toric $p$-varieties and affine $P[x]$-semimodules. We show that the category of affine toric $\p$-varieties with toric morphisms is antiequivalent to the category of affine $P[x]$-semimodules with $P[x]$-semimodule morphisms. Moreover, we show that there is a one-to-one correspondence between irreducible $T$-invariant $\p$-subvarieties of an affine toric $\p$-variety $X$ and faces of the corresponding affine $P[x]$-semimodule, where $T$ is the $\sigma$-torus of $X$. Besides, there is also a one-to-one correspondence between $T$-orbits of the affine toric $\sigma$-variety $X$ and faces of the corresponding affine $P[x]$-semimodule.

We also define projective toric $\p$-varieties in a $\D$-projective space and define abstract toric $\p$-varieties associated with a fan by gluing affine toric $\p$-varieties. It turns out that both affine toric $\p$-varieties and projective toric $\p$-varieties are abstract toric $\p$-varieties. The irreducible invariant $\p$-subvarieties-faces correspondence is generalized to abstract toric $\p$-varieties. By virtue of the correspondence theorem, we can develop a divisor theory on abstract toric $\p$-varieties and establish connections between the properties of toric $P[\sigma]$-varieties and divisor class groups.

The divisor theory for toric $P[\sigma]$-varieties developed in this paper is not complete. In algebraic geometry, many applications of the divisor theory on algebraic varieties, in particular on toric varieties, are revealed. We hope that we can give more applications of the divisor theory on toric $P[\sigma]$-varieties in the future work.

\end{document}